\documentclass[11pt,twosided,a4paper]{amsart}
\usepackage{amsmath, amsthm}
\usepackage{amsfonts,amssymb}
\usepackage{graphicx}
\usepackage{calrsfs}

\title[Hyperbolic surfaces with cone singularities]{Minimal diffeomorphism between hyperbolic surfaces with cone singularities}
\author{J\'{e}r\'{e}my Toulisse}
\address{
University of Luxembourg, Campus Kirchberg \\
Mathematics Research Unit BLG \\
6, rue Richard Coudenhove-Kalergi \\
L-1359 Luxembourg \\
Grand Duchy of Luxembourg}
\email{jeremy.toulisse@uni.lu}
\date{\today}

\newcommand{\fkp}{\mathfrak{p}}
\newcommand{\M}{\mathcal{M}_{-1}^\alpha}
\newcommand{\RR}{\mathbb{R}}

\newcommand{\ZZ}{\mathbb{Z}}
\newcommand{\Diff}{\mathcal{Diff}_0(\Sigma_\fkp)}
\newcommand{\F}{\mathcal{F}_\alpha(\Sigma_\fkp)}
\newcommand{\FF}{\mathcal{F}_{\alpha'}(\Sigma_\fkp)}
\newcommand{\T}{\mathcal{T}(\Sigma_\fkp)}
\newcommand{\Ker}{\text{Ker}}

\newcommand{\tr}{\text{tr}}
\newcommand{\E}{\mathcal{E}}

\newcommand{\z}{\overline{z}}
\newcommand{\g}{\widetilde{g}}
\newcommand{\h}{\widetilde{h}}
\newcommand{\Hi}{\mathcal{H}}
\newcommand{\D}{\mathcal{D}}
\newcommand{\Sym}{\mathit{S}}

\theoremstyle{definition}
\newtheorem{Def}{Definition}[section]
\theoremstyle{plain}
\newtheorem{theo}{Theorem}[section]
\newtheorem*{Theo}{Main Theorem}
\newtheorem*{nota}{Notations}
\newtheorem{prop}[Def]{Proposition}
\newtheorem{cor}[Def]{Corollary}
\newtheorem{lemma}[Def]{Lemma}
\theoremstyle{remark}
\newtheorem{rem}{Remark}[section]
\newtheorem*{ex}{Example}

\begin{document}

\maketitle
 
\begin{abstract}
We prove the existence of a minimal diffeomorphism isotopic to the identity between two hyperbolic cone surfaces $(\Sigma,g_1)$ and $(\Sigma,g_2)$ when the cone angles of $g_1$ and $g_2$ are different and smaller than $\pi$. When the cone angles of $g_1$ are strictly smaller than the ones of $g_2$, this minimal diffeomorphism is unique.
\end{abstract}
 
 \tableofcontents
 
\section{Introduction}

A diffeomorphism $f:(M,g_1)\longrightarrow (N,g_2)$  between two Riemannian manifolds is called \textbf{minimal} if its graph $\Gamma$ is a minimal submanifold of $(M\times N,g_1\oplus g_2)$ (that is its mean curvature tensor field vanishes everywhere).  Minimal diffeomorphisms between hyperbolic surfaces have been studied by F. Labourie and independently R. Schoen \cite{labourie},\cite{schoen}. They proved that for any two hyperbolic metrics $g_1$ and $g_2$ on $\Sigma$, there exists a unique minimal diffeomorphism $\Psi:(\Sigma,g_1)\longrightarrow (\Sigma,g_2)$ isotopic to the identity. Such a minimal diffeomorphism is also area-preserving and so its graph is a Lagrangian submanifold of $(\Sigma\times\Sigma,\omega_1\oplus(-\omega_2))$ (where $\omega_i$ is the area form associated to $g_i$); we call such a map a \textbf{minimal Lagrangian} diffeomorphism. $\Psi$ is related to harmonic maps. It is well-known (see \cite{sampson}, \cite{wolf}) that, given a conformal struture $\frak{c}$ and a hyperbolic metric $h$ on $\Sigma$, there exists a unique harmonic diffeomorphism $u: (\Sigma,\frak{c}) \longrightarrow (\Sigma,h)$ isotopic to the identity  and $h$ is characterized by the Hopf differential $\Phi(u)$ of $u$ (see Section \ref{teich} for definitions). For each pair of hyperbolic metrics on $\Sigma$, there exists a unique conformal structure $\frak{c}$ such that $\Phi(u_1)+\Phi(u_2)=0$ where $u_i: (\Sigma,\frak{c})\longrightarrow (\Sigma,g_i)$ is the unique harmonic map isotopic to the identity ($i=1,2$). Moreover, $u_2\circ u_1^{-1}$ is minimal Lagrangian and isotopic to the identity.

For an angle $\theta\in(0,2\pi)$, consider the metric obtained by gluing an angular sector of angle $\theta$ between two half-lines in the hyperbolic disk by a rotation. This metric is called \textbf{local model for hyperbolic metric with cone singularity of angle $\theta$}. For a marked surface $\Sigma_\fkp:=\Sigma\setminus \fkp$ where $\fkp=(p_1,...,p_n)\subset \Sigma$ and for $\alpha:=(\alpha_1,...,\alpha_n)\in \left(0,\frac{1}{2}\right)^n$ such that $\chi(\Sigma)+\sum_{i=1}^n(\alpha_i-1)<0$, one can construct the Fricke space of $\Sigma_\fkp$ with cone singularities of angle $\alpha$, denoted by $\F$,  as the moduli space of marked (where the marking fix each $p_i\in\frak{p}$) hyperbolic metrics on $\Sigma_\fkp$ with cone singularities of angle $2\pi\alpha_i$ at the $p_i$ (see Section \ref{teich} for the construction). In a previous paper \cite{toulisse}, we proved the existence of a unique minimal Lagrangian diffeomorphism isotopic to the identity for each pair of points $g_0,g_1\in\F$ (that is when the cone angles of $g_1$ and $g_2$ are equal). The proof of this result used the deep relations between three dimensional AdS geometry and hyperbolic surfaces: we showed the existence of a unique surface maximizing the area in some AdS singular spacetimes, and proved that it implies the existence of a unique minimal Lagrangian map (see \cite{toulisse} for more details). In this paper, we address the question of the existence and uniqueness of minimal diffeomorphism between hyperbolic cone surfaces with different cone angles. In particular, we prove:

\begin{Theo}
Given $\alpha,\alpha' \in \left( 0,\frac{1}{2} \right)^n,~g_1\in \F$ and $g_2\in \FF$, there exists a minimal diffeomorphism $\Psi: (\Sigma_\fkp,g_1)\longrightarrow (\Sigma_\fkp,g_2)$ isotopic to the identity. If moreover for all $i\in\{1,...,n\},~\alpha_i<\alpha'_i$ then $\Psi$ is unique.
\end{Theo}

The proof of this result is totally different from the proof in \cite{toulisse}. Here, we study the energy functional over $\T$, the Teichm\"uller space of $\Sigma_\fkp$. In his thesis \cite{jesse}, J. Gell-Redman  proved the existence of a unique harmonic diffeomorphism isotopic to the identity from a conformal surface to a surface endowed with a negatively curved metric with cone singularities of angles less than $\pi$. So, given a hyperbolic metric $h$ with cone singularites of angle $2\pi\alpha_i$ at the $p_i$, we can define the energy functional $\E_h: \T\longrightarrow \mathbb{R}$ which associates to a conformal structure on $\Sigma_\fkp$ the energy of the unique harmonic diffeomorphism $u: (\Sigma_\fkp,\frak{c}) \longrightarrow (\Sigma_\fkp,h)$  provided by \cite{jesse}. This functional only depends on the class of $h$ in $\F$.

In Section \ref{teich}, we give a precise definition of hyperbolic surfaces with cone singularities and construct $\F$.

In Section \ref{energyfunctional}, we define and study the energy functional. In particular, we prove that $\E_h$ is a proper function whose gradient at a point $\frak{c}\in\T$ is given by minus two times the real part of the Hopf differential of the harmonic map $u: (\Sigma_\fkp,\frak{c})\to (\Sigma_\fkp,h)$.

In Section \ref{minimal}, we prove the Main Theorem. To each local critical point of $\E_{g_1}+\E_{g_2}$, we construct a minimal diffeomorphism from $(\Sigma_\fkp,g_1)$ to $(\Sigma_\fkp,g_2)$. Uniqueness comes from stability of minimal graphs in $(\Sigma_\fkp\times\Sigma_\fkp,g_1\oplus g_2)$ which follows  from an application of the maximum principle to elliptic PDE satisfied by the harmonic diffeomorphisms.

It would be interesting to study the possible relations between the minimal map of the Main Theorem and AdS geometry. In particular, this minimal map should be related to some ``maximal'' surface in some AdS manifold with spin particles (as introduced in \cite{barbot} in the Minkowski case). We leave this question for a future work.

\textbf{Aknowledgment:} I would like to thank J.-M. Schlenker for valuable discussions about the subject.

\section{Fricke space with cone singularities}\label{teich}

\subsection{Hyperbolic disk with cone singularity}\label{sing}

Let $\alpha\in(0,1)$ and $\mathbb{H}^2:=(\mathbb{D}^2,g_{p})$ be the unit disk equipped with the Poincar\'{e} metric. Cut $\mathbb{D}^2$ along two half-lines making an angle $2\pi\alpha$ intersecting at the center $0$ of $\mathbb{D}^2$ and define $\mathbb{H}^2_\alpha$ as the space obtained by gluing the boundary of the angular sector of angle $2\pi\alpha$ by a rotation fixing $0$. Topologically, $\mathbb{H}^2_\alpha=\mathbb{D}^2\setminus\{0\}$ and the induced metric $g_\alpha$ (which is not complete) is hyperbolic outside $0$ and carries a conical singularity of angle $2\pi\alpha$ at $0$. We call $\mathbb{H}^2_\alpha=(\mathbb{D}^*,g_\alpha)$ the \textbf{hyperbolic disk with cone singularity of angle $2\pi\alpha$}.

In conformal coordinates, we have the well-known expression: 
$$g_{p}=\frac{4}{(1-\vert \widetilde{z} \vert^2)^2}\vert d\widetilde{z}\vert^2.$$
Using the coordinates $\widetilde{z}=\frac{1}{\alpha}z^\alpha$, we obtain:
$$g_\alpha=\frac{4\vert z \vert^{2(\alpha-1)}}{(1-\alpha^{-2}\vert z \vert^{2\alpha})^2}\vert dz \vert^2.$$
In cylindrical coordinates $(\rho,\theta)\in\mathbb{R}_{>0}\times \RR/2\pi\alpha\ZZ$, we have:
$$g_\alpha=d\rho^2+\sinh^2\rho d\theta^2.$$

\subsection{Hyperbolic surface with cone singularities}\label{hyperbolicconesurfaces}
Here we define the moduli space of hyperbolic metrics with cone singularities. Before that, we need to introduce weighted H\"older spaces adapted to the study of metrics with conical singularities and to the existence of harmonic maps (see \cite[Section 2.2]{jesse}):
\begin{Def}
For $R>0$, let $D(R):=\{ z\in\mathbb{C},~\vert z\vert \in (0,R)\}$. We say that a function $f: D(R)\longrightarrow \mathbb{C}$ is in $\chi^{0,\gamma}_b(D(R))$ with $\gamma\in (0,1)$ if, writing $z=re^{i\theta}$ and $z'=r'e^{i\theta'}$
$$\Vert f\Vert_{\chi^{0,\gamma}_b}:=\underset{{D(R)}}{\sup}\vert f\vert+\underset{z,z'\in D(R)}{\sup}\frac{\vert f(z)-f(z')\vert}{\vert \theta-\theta'\vert^\gamma +\frac{\vert r-r'\vert^\gamma}{\vert r+r'\vert^\gamma}}<+\infty.$$
We say that $f\in \chi^{k,\gamma}_b(D(R))$ if, for all linear differential operator $L$ of order $k$, $L(f)\in \chi^{0,\gamma}_b(D(R))$ (note that in particular, $f\in \mathcal{C}^k(D(R))$).
\end{Def}
\textbf{From now and so on, all the cone angles will be considered strictly smaller than $\pi$.}

Let $\Sigma$ be a closed oriented surface, $\fkp=(p_1,...,p_n)\subset \Sigma$ be a set of points. Denote by $\Sigma_\fkp:=\Sigma\setminus \fkp$ and let $\alpha:=(\alpha_1,...,\alpha_n)\in \left(0,\frac{1}{2}\right)^n$ be such that $\chi(\Sigma_\fkp)-\sum_{i=1}^n(\alpha_i-1)<0$ (this condition implies the existence of hyperbolic metric with cone singularities). 

\begin{Def}
A hyperbolic metric on $\Sigma_\fkp$ with cone singularities of angle $\alpha$ is a metric $g$ so that
\begin{itemize}
\item[-] For each compact $K\subset \Sigma_\fkp$, $g_{\vert K}$ is $\mathcal{C}^2$ and has constant curvature $-1$, 
\item[-] for each puncture $p_i\in \fkp$, there exists a neighborhood $U$ with local conformal coordinates $z$ centered at $p_i$ together with a local diffeomorphism $\psi\in \chi^{2,\gamma}_b(U)$ so that
$$g_{\vert U}= \psi^*g_{\alpha_i}.$$
\end{itemize}
We denote by $\M$ the space of such metrics.
\end{Def}

\begin{rem}\label{generalsingularity}
In the general case, one says that a metric $g$ on $\Sigma_\fkp$ has a conical singularity of angle $2\pi\alpha$ at $p\in \fkp$ if in a neighborhood of $p$, $g$ has the form
$$g=e^{2\lambda}\vert z\vert^{2(\alpha-1)}\vert dz\vert^2,$$
where $\lambda$ is a continuous function (see \cite{troyanov}).
\end{rem}

\begin{Def}
Let $\Diff$ be the space of diffeomorphisms $\psi$ of $\Sigma_\fkp$ isotopic to the identity (in the isotopy class fixing each $p_i\in\fkp$) so that, for each compact $K\subset \Sigma_\fkp$, $\psi_{\vert K}$ is of class $\mathcal{C}^3$ and, for each marked point $p_i\in \fkp$, there exists an open neighborhood $U$ so that $\psi\in \chi^{2,\gamma}_b(U)$.
\end{Def} 
Note that, $\Diff$ acts by pull-back on $\M$ and the quotient space $\F:=\M/\Diff$ is a smooth manifold called the \textbf{Fricke space with cone singularities of angles $\alpha$}.

\begin{rem}
The regularity condition around the punctures we impose is the one we need to use the Theorem of \cite{jesse} providing the existence of a harmonic diffeomorphism isotopic to the identity.
\end{rem}

\begin{prop}\label{boundeddistance}
For a fixed $\alpha\in \left(0,\frac{1}{2}\right)^n$ and all $i\in\{1,...,n\}$, there exists $r_i>0$ such that for each hyperbolic metric with cone singularities $g\in \M$ the open set $V_i:=\{x\in \Sigma_\fkp,~d_g(x,p_i)<r_i\}$ is isometric to a neighborhood of $0$ in $\mathbb{H}^2_{\alpha_i}$ (here $d_g(.,.)$ is the distance w.r.t. $g$).
\end{prop}

\begin{proof}
The result follows from the fact that the distance between two conical singularities of angles less than $\pi$ on a hyperbolic surface is bounded from below.

Let $p_1$ and $p_2$ two conical singularities of angles $2\pi\alpha_1<\pi$ and $2\pi\alpha_2<\pi$ respectively on a hyperbolic cone surface. Let $\beta$ be an embedded geodesic segment joining $p_1$ and $p_2$, and denote by $\gamma$ the unique geodesic in a regular neighborhood of $\beta$ homotopic to a simple closed curve around $p_1$ and $p_2$. Finally, denote by $\delta_i$ the geodesic arc from $p_i$ making an angle $\pi\alpha_i$ with $\beta$ ($i=1,2$).

We claim that, as $2\pi\alpha_1$ and $2\pi\alpha_2$ are (strictly) smaller than $\pi$, the distance between $\beta$ and $\gamma$ is strictly positive. In fact, take a regular neighborhood $U$ of $\beta$, and cut it along $\beta$, $\delta_1$ and $\delta_2$. We get two connected components $V$ and $W$, each containing $\beta,~\delta_1$ and $\delta_2$ in their boundary. By a hyperbolic isometry, send $V$ to the upper half-plane model of $\mathbb{H}^2$, sending $\beta$ on the imaginary axis. Denote by $N$ the unit (for the Euclidean metric) vector field orthogonal to $\beta$ pointing to the interior of $V$. Note that $N$ is a Jacobi field. For $\epsilon>0$ small enough, the length of the geodesic arc $\beta_\epsilon:=\exp(\epsilon N)\cap V$ is strictly smaller than the length of $\beta$ (see Figure \ref{beta}). It implies that if $\gamma$ is too close to $\beta$ (or even coincide), then a local deformation of $\gamma$ along the vector field $N$ would strictly decreases its length. So the distance between $\gamma$ and $\beta$ is strictly positive.

\begin{figure}[!h] 
\begin{center}
\includegraphics[height=6cm]{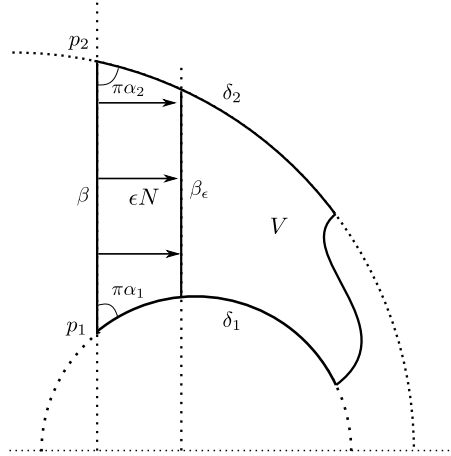}
\end{center}
\caption{The geodesic $\beta_\epsilon$} 
\label{beta}
\end{figure}

Now, consider the connected component $S$ of $\Sigma\setminus \gamma$ containing $p_1$ and $p_2$, and cut it along $\beta$, $\delta_1$ and $\delta_2$. The remaining surfaces are two isometric hyperbolic quadrilaterals (see Figure \ref{quadrilateral}). When the length of $\gamma$ tends to zero, each quadrilateral tends to a hyperbolic triangle of angles $\pi\alpha_1,~\pi\alpha_2$ and $0$. In such a triangle, the length on $\beta$ satisfies
$$\cosh(l(\beta))=\frac{1+\cos(\pi\alpha_1)\cos(\pi\alpha_2)}{\sin(\pi\alpha_1)\sin(\pi\alpha_2)}.$$
It corresponds to the lower bound for the distance between two hyperbolic cone singularities of angles $2\pi\alpha_1$ and $2\pi\alpha_2$. 

\begin{figure}[!h] 
\begin{center}
\includegraphics[height=6cm]{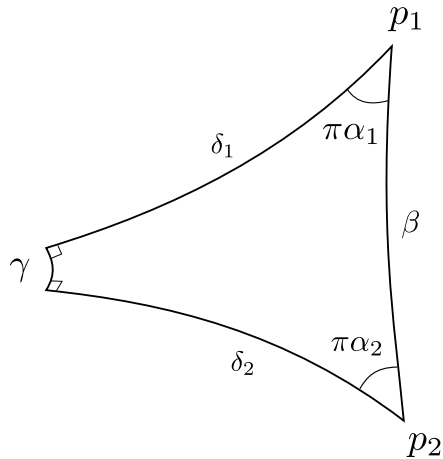}
\end{center}
\caption{Hyperbolic quadrilateral} 
\label{quadrilateral}
\end{figure}

Applying this result to the universal covering of $\Sigma_\fkp$, we get a lower bound for the injectivity radius of the singular points on a hyperbolic cone surface.
\end{proof}
From now and so on, we fix a cylindrical coordinates system $(\rho_i,\theta_i): V_i \to \mathbb{H}^2_{\alpha_i}$ centered at $p_i$ for each $i\in\{1,...,n\}$ (where the $V_i$ are as in Proposition \ref{boundeddistance}). Note that Proposition \ref{boundeddistance} implies that, up to a gauge, we can always assume that for each $i\in\{1,...,n\}$, every metric $g\in \M$ has the following expression:
$$g_{\vert U_i}=d\rho_i^2+\sinh^2\rho_id\theta_i^2.$$
We get the following Corollary:

\begin{cor}\label{gaugefixing}  Let $g_0\in \M$ and let $\widetilde{h}:=\frac{d}{dt}_{\vert_{t=0}} g_t$ be a deformation of $g_0$. There exists a vector field $v\in \text{Lie}(\Diff)$ (the Lie algebra of $\Diff$), so that
 $$\widetilde{h}=h+\mathcal{L}_vg_0,\text{ and } h_{\vert V_i}=0~~\forall i\in\{1,...,n\}.$$
Here $\mathcal{L}_vg_0$ is the Lie derivative of $g$ in the direction $v$ and the $V_i$ are defined as in Proposition \ref{boundeddistance}. We call such a $h$ a \textbf{normalized deformation}.
\end{cor}

\medskip

\textbf{Analysis on hyperbolic cone manifolds.} Let $(\Sigma_\fkp,g)$ be a hyperbolic surface with cone singularities of angle $\alpha\in\left(0,\frac{1}{2}\right)^n$. It is not obvious that classical results of geometric analysis on Riemannian manifolds (as integration by parts) extend to hyperbolic cone surfaces. In this section, we study differential operators on vector bundles over $(\Sigma_\fkp,g)$ in the framework of unbounded operators. For the convenience of the reader, we recall here basic facts about unbounded operators between Hilbert spaces. A good reference for the subject is \cite{unbounded}.

\medskip

\textbf{Unbounded operators.} Let $\Hi_1$ and $\Hi_2$ be two Hilbert spaces with scalar product $\langle.,.\rangle_1$ and $\langle .,.\rangle_2$ respectively.

\begin{Def}
An \textit{unbounded operator} is a linear map $$T: \D(T)\subset \Hi_1 \longrightarrow \Hi_2$$ where $\D(T)$ is a linear subset of $\Hi_1$ called the domain of $T$.
\end{Def}

\begin{ex}
Let $I\subset \mathbb{R}$ be an interval and $D$ an order $n\in\mathbb{N}$ linear differential operator. We see $D: \mathcal{C}_0^\infty(I)\subset L^2(I)\longrightarrow L^2(I)$ as an unbounded operator (here $\mathcal{C}^\infty_0(I)$ is the space of $\mathcal{C}^\infty$ real valued functions over $I$ with compact support). 
\end{ex}

Of course, one notes that in this example, $\mathcal{C}^\infty_0$ is probably not the biggest set (with respect to the inclusion) where $D$ can be defined. This motivates the following definitions:

\begin{Def}
Let $T_1$ and $T_2$ two unbounded operators from $\Hi_1$ to $\Hi_2$. We say that $T_1$ \textbf{extends} $T_2$ (and we denote by $T_2\subset T_1$) if $\D(T_2)\subset \D(T_1)$ and $T_{1\vert_{\D(T_2)}}=T_2$.
\end{Def}

We have the important notion of closed and closable operators:

\begin{Def}
An unbounded operator $T$ is \textbf{closed} if its graph $\mathcal{G}(T)$ is closed in $\Hi_1\oplus \Hi_2$.  $T$ is called \textbf{closable} if the closure of  $\mathcal{G}(T)$ in $\Hi_1\oplus \Hi_2$ is the graph of an unbounded operator $\overline{T}$. In this case, $\overline{T}$ is called the closure of $T$.
\end{Def}

We have the following characterization (cf. \cite[Proposition 1.5]{unbounded}):

\begin{prop}\label{characterizationclosable}
$T$ is closable if and only if, for each sequence $(x_n)_{n\in\mathbb{N}}\subset \D(T)$ such that $\lim_{n\to \infty}x_n=0$ and $(Tx_n)_{n\in\mathbb{N}}$ converges to $y\in\Hi_2$ we have $y=0$.
\end{prop}

\begin{rem}
If $T$ is continuous, $\lim_{n\to \infty}x_n=0$ implies $\lim_{n\to\infty} Tx_n=0\in \Hi_2$, and so $T$ is closable by Proposition \ref{characterizationclosable}. For $T$ being closable, we just require that if $(Tx_n)_{n\in\mathbb{N}}$ converges in $\Hi_2$, then it converges to the ``good'' limit. Hence closability condition can be thought as a weakening of continuity.
\end{rem}

Using the scalar products of $\Hi_1$ and $\Hi_2$, we can define the adjoint of an unbounded operator with dense domain:

\begin{Def}
Let $T: \D(T)\subset\Hi_1\longrightarrow \Hi_2$ be an unbounded operator such that $\D(T)$ is dense in $\Hi_1$. We define the \textbf{adjoint} of $T$ as the unbounded operator $T^*: \D(T^*)\subset \Hi_2 \longrightarrow \Hi_1$ where:
$$\D(T^*):=\{y\in \Hi_2,~\text{there exists }u\in \Hi_1 \text{ such that }\langle Tx,y\rangle_2=\langle x,u \rangle_1,~\forall x\in \D(T)\}.$$
As $\D(T)$ is dense, $u$ is uniquely defined and we set $T^*y:=u$.
\end{Def}

Determining the domain of an adjoint operator is generally difficult. Hence we have the notion of a formal adjoint:

\begin{Def}
Let $T$ be an unbounded operator with dense domain. We say that an operator $T^t: \D(T^t) \subset \Hi_2 \longrightarrow \Hi_1$ is a \textbf{formal adjoint} of $T$ is for all $x\in \D(T),~y\in \D(T^t)$ we have $\langle Tx,y\rangle_2=\langle x,T^ty\rangle_1$.
\end{Def}

\begin{rem}
Note that, by Riesz' theorem, $y\in \D(T^*)$ if and only if the application $x\longmapsto \langle Tx,y\rangle$ is continuous on $\D(T)$. In particular, for every formal adjoint $T^t$ of $T$, we have $\D(T^t)\subset \D(T^*)$ and by density $T^*_{\vert \D(T^t)}=T^t$. So $T^*$ extends every formal adjoint of $T$.
\end{rem}

We have the following classical properties (see e.g. \cite[Chapter 1]{unbounded}):

\begin{prop}\label{unboundedop}
Let $S$ and $T$ be two unbounded operators from $\Hi_1$ to $\Hi_2$ with dense domain. Then:
\begin{itemize}
\item[i.] $T^*$ is closed.
\item[ii.] If $T\subset S$ then $S^*\subset T^*$.
\item[iii.] $\D(T^*)$ is dense if and only if $T$ is closable. In this case, $\overline{T}=T^{**}$.
\item[iv.] $\Im(T)=\Ker(T^*)^\bot$ (where $\Im$ and $\Ker$ design the image and the kernel respectively).
\end{itemize}
\end{prop}

\medskip

\textbf{Application to geometric analysis on cone surfaces.} Let $E$, $F$ be two vector bundles over a hyperbolic cone surface $(\Sigma_\fkp,g)$ (recall that the cone angles are supposed strictly smaller than $\pi$), and equip $E$ and $F$ with Riemannian metrics $(.,.)_E$ and $(.,.)_F$ respectively. For $k\in\mathbb{N}$, denote by $\mathcal{C}^k_0(E)$ (respectively $\mathcal{C}^k(E)$ and $L^2(E)$) the space of sections of $E$ which are $\mathcal{C}^k$ with compact support (respectively $\mathcal{C}^k$ and $L^2$). The Riemannian metric on $E$ turns $L^2(E)$ into a Hilbert space with respect to the following scalar product:

$$\langle f,g\rangle_E:=\int_{\Sigma_\fkp} (f,g)_E vol_g.$$
Note that $\mathcal{C}^\infty_0(E)\subset L^2(E)$ is a dense subset.

\begin{nota} Denote by $T^{(r,s)}\Sigma_\fkp$ the bundle of $(r,s)$-tensors (that is $r$-covariant and $s$-contravariant) over $\Sigma_\fkp$ and by $\Sym^k\Sigma_\fkp\subset T^{(k,0)}\Sigma_\fkp$ the bundle of $k$-symmetric tensors. The metric $g$ on $\Sigma_\fkp$ induces a metric on these bundles, also denoted by $g$.
\end{nota}

We need some results of integration by parts in cone manifolds. Some good references for this theory are \cite{cheeger},\cite[Part 3]{gregoire} and \cite{thesegregoire}.

\medskip

\textbf{Operators on covariant tensors.} We denote by $\mathring{\nabla}$ the covariant derivative associated to $g$. We see $\mathring{\nabla}$ as an unbounded operator:
$$\mathring{\nabla}: \D(\nabla):=\mathcal{C}^1_0\left(T^{(r,0)}\Sigma_\fkp\right)\subset L^2\left(T^{(r,0)}\Sigma_\fkp\right) \longrightarrow L^2\left(T^{(r+1,0)}\Sigma_\fkp\right).$$
Stokes formula for compactly supported tensors implies that $\mathring{\nabla}$ admits a formal adjoint
$$\nabla^t: \D(\nabla^t)=\mathcal{C}^1_0\left( T^{(r+1,0)}\Sigma_\fkp\right) \subset L^2\left(T^{(r+1,0)}\Sigma_\fkp\right)\longrightarrow L^2\left(T^{(r,0)}\Sigma_\fkp\right),$$
where 
$$\nabla^t\eta(X_1,...,X_r)=-\sum_{i=1}^2 (\nabla_{e_i}\eta)(e_i,X_1,...,X_r),$$
for $(e_1,e_2)$ an orthonormal framing of $T\Sigma_\fkp$.

As $\mathcal{C}^\infty_0\left(T^{(r+1,0)}\Sigma_\fkp\right)\subset \D(\nabla^t)$ and $\nabla^t\subset \nabla^*$ (here $\nabla^*$ is the adjoint of $\mathring{\nabla}$), then $\mathring{\nabla}$ is closable (by Proposition \ref{unboundedop}). Denote by $\nabla$ its closure (so $\nabla=\nabla^{**}$). The restrictions of the operators $\nabla$ and $\nabla^*$ to smooth sections are described above.

\medskip

\textbf{Operators on symmetric tensors.}
For $k>0$, we define the divergence operator $\mathring{\delta}$ by 
$$\mathring{\delta}:=\nabla^*_{\vert \mathcal{C}^1_0(\Sym^k\Sigma_\fkp)}.$$
Again, Stokes formula for compactly supported symmetric tensors implies that $\mathring{\delta}$ admits a formal adjoint, 
$$\delta^t: \mathcal{C}^1_0(\Sym^{k-1}\Sigma_\fkp) \subset L^2(\Sym^{k-1}\Sigma_\fkp)\longrightarrow L^2(\Sym^k\Sigma_\fkp)$$
 which is the composition of the covariant derivative with the symmetrization.

It follows that $\delta^*$ (the adjoint of $\mathring{\delta}$) has dense domain, and so $\mathring{\delta}$ is closable. We denote by $\delta$ its closure.

\begin{nota} By analogy with classical Sobolev spaces, we introduce the following notations:
\begin{itemize}
\item[-] $H^1(\Sym^1\Sigma_\fkp):=\D(\delta^*)\subset L^2(\Sym^1\Sigma_\fkp)$,
\item[-] $H^1(\Sym^2\Sigma_\fkp):=\D(\delta)\subset L^2(\Sym^2\Sigma_\fkp)$,
\item[-] $H^2(\Sym^1\Sigma_\fkp):=\D(\delta\circ\delta^*)\subset L^2(\Sym^1\Sigma_\fkp)$,
\item[-] $H^1(\Sigma_\fkp)=\D(\delta^*)\subset L^2(\Sigma_\fkp)$ (the space of $L^2$ functions over $\Sigma_\fkp$),
\item[-] $H^2(\Sigma_\fkp)=\D(\delta\circ\delta^*)\subset L^2(\Sigma_\fkp)$.
\end{itemize}
\end{nota}

We have a result of integration by parts for symmetric tensors on $(\Sigma_\fkp,g)$. The proof is analogous to the proof of \cite[Theorem 1.4.3]{gregoire}, however, as it is a central result in what follows, we include it.

\begin{theo}\label{ipp}
For all $u\in H^1(\Sym^1\Sigma_\fkp)\cap \mathcal{C}^1(\Sym^1\Sigma_\fkp)$ and $h\in H^1(\Sym^2\Sigma_\fkp)\cap \mathcal{C}^1(\Sym^2\Sigma_\fkp)$, we have:
$$\langle \delta^*u,h\rangle_{\Sym^2}=\langle u,\delta h\rangle_{\Sym^1}.$$
For all $f\in\mathcal{C}^1(\Sigma_\fkp)\cap H^1(\Sigma_\fkp)$ and $\alpha\in\mathcal{C}^1(\Sym^1\Sigma_\fkp)\cap H^1(\Sym^1\Sigma_\fkp)$,
$$\langle \delta^*f,\alpha\rangle_{\Sym^1}=\langle f,\delta\alpha\rangle_{L^2(\Sigma_\fkp)}.$$
\end{theo}

\begin{proof}
The proof of the two statements are analogous, so we just prove the first one (which is a little bit more technical).

Let's prove the result when $(\Sigma_\fkp,g)$ contains a unique cone singularity $p$ of angle $2\pi\alpha$. To prove the result in the general case, we just apply the following computation to each puncture.

Fix cylindrical coordinates $(\rho,\theta)\in (0,r)\times \mathbb{R}/2\pi\alpha\mathbb{Z}$ in a neighborhood of $p$ so that 
$$g_{\vert V}=d\rho^2+\sinh^2\rho d\theta^2.$$
For $t\in(0,r)$, denote by $U_t:=\{(\rho,\theta)\in V, \rho<t \}$. 

For $u\in H^1(\Sym^1\Sigma_\fkp)\cap \mathcal{C}^1(\Sym^1\Sigma_\fkp)$ and $h\in H^1(\Sym^2\Sigma_\fkp)\cap \mathcal{C}^1(\Sym^2\Sigma_\fkp)$, we have:
$$\int_{\Sigma\setminus U_t}\big(g(u,\delta h)-g(\delta^*u,h)\big)dv_g=\int_{\Sigma\setminus U_t}\left(g(u,\nabla^*h)-\frac{1}{2}\left(g(\nabla u,h)+g(F\circ \nabla u,h)\right)\right)dv_g,$$
where $F: T^{(2,0)}\Sigma_\fkp\longrightarrow T^{(2,0)} \Sigma_\fkp$ is defined by $F\eta (x,y):=\eta (y,x)$. Note that, for $\theta,\eta\in L^2\left(T^{(2,0)}\Sigma_\fkp\right)$,
$$\langle F\theta,\eta\rangle_{T^{(2,0)}}=\langle \theta,F\eta\rangle_{T^{(2,0)}}.$$
As $h$ is symmetric, and applying Stokes formula, we get:
$$\int_{\Sigma\setminus U_t}(g(u,\delta h)-g(\delta^*u,h))dv_g=\int_{\Sigma\setminus U_t}(g(u,\nabla^*h)-g(\nabla u,h))dv_g=\int_{\partial U_t}g_{\vert\partial U_t}(u,i_{e_\rho}h)dv_g,$$
where $i_{e_\rho}h=h(e_\rho,.)$ and $e_\rho=\partial_\rho$ is the unit vector field normal to $\partial U_t$.

As $t$ tends to $0$, the left hand side tends to $\langle u,\delta h\rangle_{\Sym^1}-\langle \delta^*u,h\rangle_{\Sym^2}$. Denote by $I_t$ the right hand side. By Cauchy-Schwarz inequality,
$$\vert I_t\vert \leq \int_{\partial U_t}\vert u\vert \vert i_{e_\rho}h\vert dv_g\leq \left(\int_{\partial U_t}\vert u \vert^2dv_g \right)^{1/2}\left(\int_{\partial U_t}\vert i_{e_\rho}h\vert^2dv_g \right)^{1/2}.$$
When $u\neq 0$, $\vert u \vert$ is differentiable and $d\vert u \vert(x)=g\left(\nabla_x u, \frac{u}{\vert u \vert}\right)$, so we set
$$\partial_\rho \vert u \vert = g\left(\nabla_{e_\rho}u,\frac{u}{\vert u \vert}\right);$$
and if $u=0$, set $\partial_\rho \vert u \vert =0$. Note that $\partial_\rho \vert u \vert$ is the partial derivative of $\vert u\vert$ is the sense of distributions. In fact, for all $t,a\in(0,r)$ and $\theta$ fixed, we have
$$\vert u(t,\theta)\vert-\vert u(a,\theta)\vert=\int_a^t \partial_\rho\vert u(\rho,\theta)\vert d\rho .$$
In particular, as $\left\vert \partial_\rho \vert u \vert \right\vert\leq \vert \nabla_{e_\rho} u\vert$,
$$\vert u(t,\theta)\vert \leq \vert u(a,\theta)\vert+\int_t^a\vert \nabla_{e_\rho} u\vert d\rho.$$
So
$$\vert u(t,\theta)\vert^2\leq 2\vert u(a,\theta)\vert^2+2\left( \int_t^a\vert \nabla_{e_\rho}u\vert d\rho \right)^2.$$
Applying Cauchy-Schwarz, we obtain
\begin{eqnarray*}
\left( \int_t^a \vert \nabla_{e_\rho}u\vert d\rho \right)^2 & \displaystyle{\leq \int_t^a \frac{d\rho}{\rho}\int_t^a \rho\vert \nabla_{e_\rho}u\vert^2d\rho} \\
&\displaystyle{\leq \left\vert \ln \left(\frac{t}{a}\right)\right\vert\int_t^a \rho\vert \nabla_{e_\rho}u\vert^2d\rho}.
\end{eqnarray*}
Finally, we get
\begin{eqnarray*}
\int_{\partial U_t}\vert u \vert^2dv_g & \leq & \displaystyle{2\int_{\partial U_t} \vert u(a)\vert^2dv_g + \int_{\partial U_t}\left(2\left\vert \ln(\frac{t}{a})\right\vert\int_t^a\rho\vert \nabla_{e_\rho}u\vert^2d\rho \right)dv_g} \\
& \leq & \displaystyle{2t\int_{\theta=0}^{2\pi\alpha}\vert u(a,\theta)\vert^2d\theta+2\left\vert \ln(\frac{t}{a})\right\vert\int_{\partial U_t}\left(\int_t^a \rho\vert\nabla_{e_\rho} u\vert^2d\rho\right)dv_g} \\
& \leq & \displaystyle{2t\int_{\theta=0}^{2\pi\alpha}\vert u(a,\theta)\vert^2d\theta + 2t\left\vert \ln(\frac{t}{a})\right\vert \int_{\theta=0}^{2\pi\alpha}\int_t^a \vert \nabla_{e_\rho} u\vert^2\rho d\rho d\theta} \\
& \leq & \displaystyle{2t\int_{\theta=0}^{2\pi\alpha}\left\vert u(a,\theta)\right\vert^2d\theta + 2t\left\vert \ln(\frac{t}{a})\right\vert \int_{U_a} \vert \nabla_{e_\rho}u\vert^2dv_g} \\
& = & O(t\ln t).
\end{eqnarray*}
Now, as $h\in L^2(\Sym^2\Sigma_\fkp)$,
$$\int_0^a\left( \int_{\partial U_t}\vert i_{e_\rho}h\vert^2 dv_g\right)\leq \int_0^a\left(\int_{\partial U_t}\vert h\vert^2dv_g\right)=\int_{U_a}\vert h\vert^2dv_g < +\infty,$$
that is, the function $t\longmapsto \int_{\partial U_t}\vert h \vert^2$ is integrable on $(0,a)$. As the function $(t\ln t)^{-1}$ is not integrable in $0$, there exists a sequence $(t_n)_{n\in\mathbb{N}}$ with $t_n \rightarrow 0$ such that
$$\int_{\partial U_{t_n}}\vert h\vert^2dv_g=o((t_n\ln t_n)^{-1}).$$
It follows that $\underset{n\to\infty}{\lim} I_{t_n}=0$.
\end{proof}

We have a very useful corollary:

\begin{cor}\label{selfadjoint}
For $i=1,2$, the operator $\delta\delta^*: H^2(\Sym^i\Sigma_\fkp)\longrightarrow L^2(\Sym^i\Sigma_\fkp)$ is self-adjoint with strictly positive spectrum.
\end{cor}

\begin{proof}
The fact that $\delta\delta^*$ are self-adjoint follows directly from Theorem \ref{ipp}. Let $\lambda\geq 0$ such that, for $f\in H^2(\Sym^i\Sigma_\fkp)$ ($i=1,2$),
$$\delta\delta^*f+\lambda f=0.$$
Taking the scalar product with $f$, and using Proposition \ref{ipp}, we get:
$$\langle \delta\delta^* f+\lambda f,f\rangle_{\Sym^i}=\| \delta^*f\|^2_{\Sym^{i+1}}+\lambda\|f\|_{\Sym^i}^2=0,$$
and so $f=0$.
\end{proof}

\subsection{Tangent space to $\F$} Here we prove the following result:
\begin{prop}
For $[g_0]\in\F$, there is a natural identification of $T_{[g_0]}\F$ with the space of meromorphic quadratic differentials on $\Sigma=\Sigma_\fkp \cup \fkp$ with at most simple poles at the $p_i\in\fkp$ (where the complex structure on $(\Sigma_\fkp,g_0)$ is the one associated to $g_0$).
\end{prop}

\begin{proof}
Fix $g_0\in \M$ and let 
$$\widetilde{h}=\frac{d}{dt}_{\vert_{t=0}} g_t\in T_{g_0}\M,$$
 where $(g_t)_{t\in I}$ is a smooth path in $\M$ with $g_{t=0}=g_0$ (and $0\in I\subset \mathbb{R}$ is an interval). By Corollary \ref{gaugefixing}, there exists a vector field $v\in \text{Lie}(\Diff)$ (the Lie algebra of $\Diff$), so that
 $$\widetilde{h}=h+\mathcal{L}_vg,~h_{\vert V_i}=0~\forall i\in\{1,...,n\}.$$
Note that in particular, $h\in\mathcal{C}^2_0(\Sym^2\Sigma_\fkp)$.
 
Such a symmetric 2-tensor $h$  on $\Sigma_\fkp$ is tangent to the space $\M$ of hyperbolic metrics with cone singularities if and only if the differential of the sectional curvature $dK_{g_0}$ in the  direction $h$ is equal to $0$.
 
First, we have a canonical orthogonal splitting:

\begin{lemma}
For all normalized deformation $h\in T_{g_0} \M$, there exists $u\in H^2(\Sym^1\Sigma_\fkp)$ and $h_0\in H^1(\Sym^2\Sigma_\fkp)$ with $\delta h_0=0$ such that:
$$h=h_0+\mathcal{L}_{u^\sharp}g_0,$$
where $u^\sharp$ is the vector field dual to $u$. Moreover, this splitting is orthogonal with respect to the scalar product of $L^2(\Sym^2\Sigma_\fkp)$.
\end{lemma}

\begin{proof}
As $h\in \mathcal{C}^2_0(\Sym^2 \Sigma_\fkp)$, $\delta h\in \mathcal{C}^1_0(\Sym^1\Sigma_\fkp)\subset L^2(\Sym^1\Sigma_\fkp)$. So we want to find $u\in H^2(\Sym^1\Sigma_\fkp)$ so that
\begin{equation}\label{split}
2\delta\delta^*u=\delta h.
\end{equation}
It is possible to solve (\ref{split}) if and only if $\delta h\in\Im(\delta\delta^*)$ (where $\Im$ stands for the image). 

By Corollary \ref{selfadjoint}, $\delta\delta^*$ is self-adjoint, so $\Im(\delta\delta^*)=\Ker(\delta\delta^*)^\bot$ (cf. Proposition \ref{unboundedop}). Hence we can solve (\ref{split}) if and only if $\delta h$ is orthogonal to the kernel of $\delta\delta^*$.

Take $w\in\Ker(\delta\delta^*)\subset H^2(\Sym^1\Sigma_\fkp)$. By elliptic regularity, such a $w$ is smooth. So, by Theorem \ref{ipp}, we get:
$$\langle \delta\delta^*w,w\rangle_{\Sym^1}=0=\langle \delta^*w,\delta^*w\rangle_{\Sym^2}.$$
In particular, $\delta^*w=0$, and we obtain:
$$\langle \delta h,w\rangle_{\Sym^1}=\langle h,\delta^* w\rangle_{\Sym^2}=0.$$
So $\delta h\in \Im(\delta\delta^*)$ and we can solve (\ref{split}).

Now, such a solution $u$ is smooth (at least $\mathcal{C}^4$), so we know the expression of $\delta^*u$. We have:

$$\delta^* u(x,y)=\frac{1}{2}\big((\nabla_xu)(y)+(\nabla_yu)(x)\big)=\frac{1}{2}\left( g_0(\nabla_xu^\sharp,y)+g_0(x,\nabla_yu^\sharp)\right),$$
which is the expression of $\frac{1}{2}\mathcal{L}_{u^\sharp}g_0$. In particular, setting $h_0:=h-\frac{1}{2}\delta^*u$, we get the decomposition.

Note that, if $u_1$ and $u_2$ are two solutions of (\ref{split}), they satisfy
$$\delta\delta^*(u_1-u_2)=0.$$
By integration by parts, we get that $\delta^*u_1=\delta^*u_2$. In particular, $\mathcal{L}_{u_1^\sharp}g_0=\mathcal{L}_{u_2^\sharp}g_0$, so the decomposition is independent on the choice of the solution of (\ref{split}).

Now we prove the orthogonal splitting. Let $u$ and $h_0$ as above. As such sections are smooths, we have:
$$\langle \mathcal{L}_{u^\sharp}g_0,h\rangle_{\Sym^2}=2\langle \delta^*u,h_0\rangle_{\Sym^2}=\langle u,\delta h_0\rangle=0.$$
\end{proof}

We explicit now the condition $dK_{g_0}(\widetilde{h})=0$. We have the well-know formula (e.g. \cite[Formula 1.5 p.33]{tromba}):
$$dK_{g_0}(\widetilde{h})= \delta\delta^*_{g_0}(\tr_{g_0}\widetilde{h})+\delta\delta \widetilde{h}+\frac{1}{2}\tr_{g_0}\widetilde{h},$$
where $\tr_{g_0}$ is the trace with respect to the metric $g_0$.

Applying this formula to the divergence-free part $h_0$ (which is transverse to the fiber of the projection), we get
$$\delta\delta^*_{g_0}(\tr_{g_0}h_0)+\frac{1}{2}\tr_{g_0}h_0=0.$$
By Corollary \ref{selfadjoint}, we get $tr_{g_0}h_0=0$. Moreover, one easily checks that each $h\in H^2(\Sym^2\Sigma_\fkp)\cap \mathcal{C}^2(\Sym^2\Sigma_\fkp)$ such that $\delta h=0$ and $tr_{g_0}h=0$ defines a tangent vector to $\F$ at $[g_0]$. So, we get the following identification
$$T_{[g_0]}\F=\left\{h\in H^2(\Sym^2\Sigma_\fkp)\cap \mathcal{C}^2(\Sym^2\Sigma_\fkp),~\delta h=0\text{ and }\tr_{g_0}h=0\right\}.$$
But we can go further. For $(dx,dy)$ an orthonormal framing of $T^*\Sigma_\fkp$, write
$$h_0= u(x,y)dx^2-v(x,y)(dxdy+dydx)+w(x,y)dy^2.$$
The condition $tr_{g_0}h=0$ implies $w(x,y)=-u(x,y)$. Write $(\partial_x,\partial_y)$ the framing dual to $(dx,dy)$. Let us explicit the divergence-free condition:
\begin{eqnarray*}
0 & = & \delta h(\partial_x) \\
 & = & -(\nabla_{\partial_x}h)(\partial_x,\partial_x)-(\nabla_{\partial_y}h)(\partial_y,\partial_y) \\
 & = & -\partial_x u + \partial_y w.
\end{eqnarray*}
In the same way, we get:
$$0=\delta h(\partial_y)=\partial_x v + \partial_y u.$$
These are the Cauchy-Riemann equations. It implies in particular that $f=u+iv$ is holomorphic on $\Sigma_\fkp$.

Now, for $z=x+iy,~dz=dx+idy$, set $\psi=f(z)dz^2$. It is a holomorphic quadratic differential on $\Sigma_\fkp$ such that $h=\Re(\psi)$. It follows that $\psi$ is meromorphic on $\Sigma$ with possible poles at the $p_i\in\fkp$.

We claim that, as $h=\Re(\psi)\in L^2\left(\Sym^2\Sigma_\fkp\right)$, the poles of $\psi$ at the $p_i$ are at most simple. In fact, let $p\in\fkp$ be a cone singularity of angle $2\pi\alpha$, $z$ be a local holomorphic coordinates around $p$ and $$\psi(z)=\left(\frac{a}{z^n}+g(z)\right)dz^2$$
for $a\in \mathbb{C}^*$, $n\geq 0$ and $g$ meromorphic so that $z^ng(z)\underset{z\to 0}{\longrightarrow} 0$.

It follows from Proposition \ref{boundeddistance} that around $p$, each lifting $g_0\in \M$  of $[g_0]\in \F$ is isometric to the expression $g_\alpha$ given in section \ref{sing}. In particular,
$$\psi\overline{\psi}=\left(O(\vert z\vert^{-2n}\right)\vert dz\vert^4,$$
so
$$g_0(\psi,\overline{\psi})(z)=O\left(\vert z \vert^{2(2-2\alpha-n)}\right).$$
It follows, 
$$g_0(\psi,\overline{\psi})dv_{g_0}=O\left(\vert z\vert^{2(1-\alpha-n)}\right)\vert dz \vert^2.$$
As $\alpha\in \left(0,\frac{1}{2}\right),~g_0(\psi,\overline{\psi})dv_{g_0}$ is integrable in $0$ is and only if $n\leq 1$, and the same is true for $h$.

On the other hand, given a meromorphic quadratic differential $\psi$ with at most simple poles at the $p_i$, its real part $h=\Re(\psi)$ is a zero trace divergence-free symmetric $(2,0)$ tensor in $L^2(\Sym^2\Sigma_\fkp)$. Hence, as it is smooth on $\Sigma_\fkp$, $h\in T_{[g_0]}\F$.
\end{proof}

\medskip

\textbf{A Weil-Petersson metric on $\F$.} Let $h,k\in T_{[g_0]}\F$. Fix a lifting $g_0\in \M$ of $[g_0]$. It follows from the above construction that there exists a unique lifting $\widetilde{h},\widetilde{k}\in T_{g_0}\M$ of $h$ and $k$ respectively which are divergence-free symmetric tensors of zero trace. We call such a lifting a \textbf{horizontal lifting}. Define:

$$\frac{1}{8}\langle h,k\rangle_{WP_\alpha}:=\langle \widetilde{h},\widetilde{k}\rangle_{\Sym^2}.$$

Obviously, $\langle.,.\rangle_{WP_\alpha}$ is a metric on $\F$. This metric is analogous to the metric defined in the non-singular case by  A.E. Fischer and A.G. Tromba (see \cite{fischertromba}). They proved in \cite[Theorem (0.8)]{fischertromba} that this metric coincides with the Weil-Petersson metric, so we call it \textbf{Weil-Petersson metric with cone singularities of angle $\alpha$}. 

\begin{rem}
In \cite{schumachertrapani}, the authors proved that $\langle.,.\rangle_{WP_\alpha}$ is a K\"ahler metric. It seems possible, by using the renormalized volume of quasi-Fuchsian manifolds with particles to prove that $\langle.,.\rangle_{WP_\alpha}$ admits a K\"ahler potential (see \cite{krasnovschlenkerrenormalized,krasnovschlenkersurvey}).
\end{rem}

\medskip

\textbf{Uniformization.} Here, we recall a fundamental result proved by R.C. McOwen \cite{mcowen} and independently M. Troyanov \cite{troyanov}. Let $\T$ be the Teichm\"uller space of $\Sigma_\fkp$, that is the moduli space of marked conformal structures on $\Sigma_\fkp$. We have

\begin{theo}
Given $\frak{c}\in\T$, there exists a unique $h\in \F$ in the conformal class $\frak{c}$ as long as $\chi(\Sigma)+\sum_{i=1}^n (\alpha_i-1)< 0$ (where $\Sigma=\Sigma_\fkp\cup \fkp$).
\end{theo}

This theorem provides a family of identification $\Theta_\alpha : \T \longrightarrow \F$ for each $\alpha\in \mathbb{R}_{>0}^n$ such that $\chi(\Sigma_\fkp)+\sum_{i=1}^n (\alpha_i-1)< 0$. In particular, one can define a family $(\Theta_\alpha^*\langle.,.\rangle_{WP_\alpha})_{\alpha\in\left(0,\frac{1}{2}\right)^n}$ of Weil-Petersson metric on $\T$.

\section{Energy functional on $\T$}\label{energyfunctional}

Let $g_0\in \M$ be a hyperbolic metric with cone singularities of angle $\alpha\in\left(0,\frac{1}{2}\right)^n$. We have the following result due to J. Gell-Redman \cite{jesse}:
\begin{theo}
For each $g\in\M$, there exists a unique harmonic diffeomorphism $u:(\Sigma_\fkp,g)\longrightarrow (\Sigma_\fkp,g_0)$ in the isotopy class (fixing the each $p_i$) of the identity.
\end{theo}

Recall that a harmonic map $f: (M,g)\longrightarrow (N,h)$ between Riemannian manifolds is a critical point of the energy, where the energy of $f$ is defined as follow:
$$E(f):=\int_M e(f) vol_g,$$
and $e(f)=\frac{1}{2}\| df\|^2$ is called the \textbf{energy density of $f$}. Here, $df$ is seen as a section of $T^*M\otimes f^*TN$ with the metric $g^{*}\otimes f^*h$ ($g^*$ stands for the metric on $T^*M$ dual to $g$).

Note that, when $\dim M=2$, the energy functional only depends on the conformal class $\frak{c}$ of the metric $g$. We denote by $u_{\frak{c},g_0}$ the harmonic diffeomorphism isotopic to the identity from $(\Sigma_\fkp,\frak{c})$ to $(\Sigma_\fkp,g_0)$.

Moreover, a complex structure $J_{\frak{c}}$ on $\Sigma_\fkp$ is canonically associated to $\frak{c}$. It allows us to split each symmetric two forms on $\Sigma_\fkp$ into its $(2,0),~(1,1)$ and $(0,2)$ part.
\begin{Def}
To a diffeomorphism $u: (\Sigma_\fkp,\frak{c})\longrightarrow (\Sigma_\fkp,g_0)$, we associate its Hopf differential:
$$\Phi(u):=(u^*g_0)^{(2,0)},$$
that is the $(2,0)$ part of the pull-back by $u$ of $g_0$.
\end{Def}

\medskip

\textbf{Local expressions.} Let $u:(\Sigma_\fkp,g)\longrightarrow (\Sigma_\fkp,g_0)$ be a diffeomorphism, $z$ be local isothermal coordinates on $(\Sigma,g)$. Set $g=\rho^2(z)\vert dz\vert^2$ and $g_0=\sigma^2(u)\vert du \vert^2$. 
As usual, write $u=u^1+iu^2$ and
$$\left\{\begin{array}{llllll}
\partial_z & = & \frac{1}{2}(\partial_1-i\partial_2), & \overline{\partial}_z & = & \frac{1}{2}(\partial_1+i\partial_2) \\
dz & = & dx_1+idx_2, & d\z & = & dx_1-idx_2 \\
\partial_u & = & \frac{1}{2}(\partial_{u^1}-i\partial_{u^2}), & \overline{\partial}_u & = & \frac{1}{2}(\partial_{u^1}+i\partial_{u^2}) \\
\end{array}\right.$$
We have the following expression:
\begin{eqnarray*}
du & = & \sum_{i,j=0}^2 \partial_i u^j dx_i\otimes \partial_{u^j} \\
& = & \partial_zudz\partial_u+\partial_z \overline{u}dz\overline{\partial}_u+\overline{\partial}_zud\z\partial_u+\overline{\partial}_z\overline{u}d\z\overline{\partial}_{u}.\end{eqnarray*}
It follows that
\begin{eqnarray*}
\Phi(u) & = & u^*g_0(\partial_z,\partial_z)dz^2 \\
& = & g_0\big(du(\partial_z),du(\partial_z)\big)dz^2 \\
& = & \sigma^2(u)\partial_zu\partial_z\overline{u}dz^2.
\end{eqnarray*}
Moreover, for $g^{ij}$ the coefficient of the metric dual to $g$,
\begin{eqnarray*}
e(u) & = & \frac{1}{2}\sum_{\alpha,\beta,i,j=0}^2 g^{ij}{g_0}_{\alpha\beta}\partial_iu^\alpha\partial_ju^\beta \\
& = & \rho^{-2}(z)\sigma^2(u)\left(\vert\partial_zu\vert^2+\vert\overline{\partial}_zu\vert^2\right).
\end{eqnarray*}
In particular, we have
\begin{eqnarray*}
(u^*g_0)^{(1,1)} & = & \left(u^*g_0(\partial_z,\overline{\partial}_z)+u^*g_0(\overline{\partial}_z,\partial_z)\right)\vert dz\vert^2 \\
& = & 2g_0\big(du(\partial_z),du(\overline{\partial}_z)\big)\vert dz\vert^2 \\
& = & \sigma^2(u)(\vert\partial_zu\vert^2+\vert \overline{\partial}_zu\vert^2)\vert dz\vert^2 \\
& = & \rho^2(z)e(u)\vert dz\vert^2.
\end{eqnarray*}
Note that we get the following equation for each section $\xi$ of $T^*\Sigma_\fkp\otimes u^*T\Sigma_\fkp$ with the metric $g^*\otimes u^*g$:
\begin{equation}\label{normcomputation}
\| \xi \|^2=4\rho^2\vert \langle \xi(\partial_z),\xi(\overline{\partial}_z)\rangle\vert,
\end{equation}
where $\langle.,.\rangle$ is the scalar product with respect to the metric $g_0$.

Finally, noting that the framing $(dz\partial_u,dz\overline{\partial}_u,d\z\partial_u,d\z\overline{\partial}_u)$ of $(T^*\Sigma_\fkp\otimes u^*T\Sigma_\fkp,g^*\otimes u^*g_0)$ is orthogonal and each vector has norm $\rho^{-1}(z)\sigma(u)$, we get the following expression for the Jacobian $J(u)$ of $u$:
\begin{eqnarray*}
J(u) & =  &\text{det}_{g^*\otimes u^*g_0} \left(\begin{array}{ll}
\partial_zu & \partial_z \overline{u} \\
\overline{\partial}_zu & \overline{\partial}_z\overline{u} \\
\end{array}\right) \\
& = & \rho^{-2}(z)\sigma^2(u)\left(\vert \partial_z u\vert^2-\vert \overline{\partial}_zu\vert^2\right).
\end{eqnarray*}

\begin{rem}

\

\begin{itemize}
\item[-] As in the classical case, $\Phi(u)$ is holomorphic on $(\Sigma_\fkp,J_g)$ if and only if $u$ is harmonic. So for $u$ harmonic, $\Phi(u)$ is a meromorphic quadratic differential on $(\Sigma,J_c)$ with at most simple poles at the $p_i$ (cf. \cite[Section 5.1]{jesse}).
\item[-] We have the following expression:
$$u^*g_0=\Phi(u)+\rho^2(z)e(u)\vert dz\vert^2+\overline{\Phi(u)}.$$
Thus $\Phi(u)$ measures the difference of the conformal class of $u^*g_0$ with $\frak{c}$. 
\end{itemize}
\end{rem}

\medskip

\textbf{Energy functional}
Fixing $g_0\in\M$, we define the energy functional $\widetilde{\E}_{g_0}$ on the space of conformal structures of $\Sigma_\fkp$ by:

$$\widetilde{\E}_{g_0}(\frak{c}):=E(u_{\frak{c},g_0}).$$

\begin{prop}
The energy functional $\widetilde{\E}_{g_0}$ descends to a functional $\E_{g_0}$ on $\T$.
\end{prop}

\begin{proof}
For each diffeomorphism isotopic to the identity $f\in\Diff$, $f:(\Sigma_\fkp,f^*\frak{c})\longrightarrow (\Sigma_\fkp,\frak{c})$ is holomorphic and $E$ is invariant under holomorphic mapping (see \cite[Proposition p.126]{eells-sampson}), that is $E(u_{\frak{c},g_0})=E(f^*u_{\frak{c},g_0})$. Moreover, $f^*u_{\frak{c},g_0}=u_{f^*\frak{c},g_0}$. In fact, 
$$f^*u_{\frak{c},g_0}: (\Sigma_\fkp,f^*\frak{c})\longrightarrow (\Sigma_\fkp,g_0)$$
is harmonic. So, as $f\in\Diff$ is isotopic to the identity, uniqueness of the harmonic diffeomorphism implies $f^*u_{\frak{c},g_0}=u_{f^*\frak{c},g_0}$. So $\widetilde{\E}_{g_0}$ is $\Diff$-invariant and descends to a functional $\E_{g_0}$ on $\T$.
\end{proof}

\begin{rem}
The same argument shows that $\E_{g_0}$ only depends on the class of $g_0$ in $\F$.
\end{rem}

Now, we are going to prove the following main result:

\begin{theo}\label{energy}
The energy functional $\E_{g_0}$ is proper functional and its Weil-Petersson gradient at $[g]\in \T$ is given by $-2\Re(\Phi(u_{[g],g_0}))\in T_{[g]}\T$.
\end{theo}

\subsection{Properness of $\E_{g_0}$}
Recall that (Proposition \ref{boundeddistance}), for each $g\in\F$ and $i\in\{1,...,n\}$, there exists a neighborhood $V_i=\{x\in\Sigma_\fkp,~d(x,p_i)<r_i\}$ of $p_i$ such that
$$g_{\vert V_i}=d\rho_i^2+\sinh^2\rho_i d\theta_i^2$$
 where $(\rho_i,\theta_i)$ are fixed cylindrical coordinates on $V_i$. We can choose the $V_i$ such that $V_i\cap V_j=\emptyset$ whenever $i\neq j$. We denote $V:=\bigcup_{i=1}^n V_i$. We need an important result, corresponding to Mumford's compactness theorem for the case of hyperbolic surfaces with cone singularities. The proof is an extension of Tromba's proof in the classical case \cite{tromba}.

\begin{prop}\label{mumford}
Let $(g_k)_{k\in\mathbb{N}}\subset \M$ be such that, the length of every closed geodesic $\gamma^k\subset (\Sigma_\fkp\setminus V,g_k)$ is uniformly bounded from below by $l>0$. There exists $g\in \M$ and a sequence $(f_k)_{k\in\mathbb{N}}\subset \text{Diff}(\Sigma_\fkp)$ such that
$$f_k^*g_k\underset{\mathcal{C}^2}{\longrightarrow} g.$$
\end{prop}

\begin{proof}
Let $(g_k)_{k\in\mathbb{N}}$ be as above. It follows that there exists $\rho>0$ such that, for each $k\in \mathbb{N}$ and $x\in \Sigma_{\fkp}\setminus V$, the injectivity radius of $x$ is bigger than $\rho$ (for example, take $\rho=\min\{l,r_1,...,r_n\}$). 

Fix $R>0$ such that $R<\frac{1}{2}\rho$. As the area of $(\Sigma_\fkp\setminus V,g_k)$ is independent of $k$, there exists $N>0$ such that for each $k\in\mathbb{N}$, $N$ is the maximum number of disjoint disks of radius $\frac{R}{2}$ in $\Sigma_\fkp$.

That is, for each $k\in\mathbb{N}$, there exists $\left(x_1^k,...,x_N^k\right)\subset \Sigma_\fkp\setminus V$ such that $D_{\frac{R}{2}}\left(x_1^k\right),...,D_{\frac{R}{2}}\left(x_N^k\right),V_1,...,V_n$ are disjoints (here $D_{\frac{R}{2}}(x_i^k)\subset \Sigma_\fkp$ is the disk of center $x_i^k$ and radius $\frac{R}{2}$) and $D_R(x_1^k),...,R_R(x_N^k),V_1,...,V_n$ is a covering of $\Sigma_p$.

For each $i,j\in\{1,...,N\}$ with $D_R(x_i^k)\cap D_R(x_j^k)\neq \emptyset$, note that $x_i^k\in D_{2R}(x_j^k)$, $x_j^k\in D_{2R}(x_i^k)$ and, as $2R<\rho$, there exists isometries $\Psi^k_i$ and $\Psi^k_j$ sending $D_{2R}(x_i^k)$ (resp. $D_{2R}(x_j^k)$) to the disk $B$ of radius $2R$ centered at $0$ in $\mathbb{H}^2$.

It follows that the map $\tau_{ij}^k:=\Psi_i^k\circ (\Psi_j^k)^{-1}$ is a positive local isometry of $\mathbb{H}^2$ which uniquely extend to $\tau_{ij}^k\in PSL(2,\mathbb{R})$. Moreover, for each $k$, $$\tau_{ij}^k(\Psi_j^k(x_i^k))=\Psi_i^k(x_j^k)\in B,$$
that is $(\tau_{ij}^k)_{k\in\mathbb{N}}$ is compact. So $(\tau_{ij}^k)_{k\in\mathbb{N}}$ admits a convergent subsequence whose limit is denoted by $\tau_{ij}$.

For each $i\in\{1,...,N\}$ and $j\in\{1,...,n\}$ with $D_{2R}(x_i^k)\cap V_j\neq \emptyset$, there exists an isometry $\Psi_i^k:D_{2R}(x_i^k)\longrightarrow B\subset \mathbb{H}^2$ and $\psi_j: V_j\longrightarrow \mathbb{H}^2_{\alpha_j}$. As $\psi_i(D_{2R}(x_i^k)\cap V_j)$ is a simply connected subset of $\mathbb{H}^2_{\alpha_j}$, it is isometric to a subset of $B\subset\mathbb{H}^2$ by an isometry denoted $\Phi_j$.

Pick-up a point $y^k\in D_{2R}(x_i^k)\cap V_j$. The map $\alpha^k_{ij}:=\Phi_j\circ\psi_j\circ(\Psi^k_i)^{-1}$ (see Figure \ref{alpha}) is a positive local isometry of $\mathbb{H}^2$ which uniquely extends to an element of $PSL(2,\mathbb{R})$. Moreover, $\alpha_{ij}^k$ sends $\Psi_i^k(y)$ to $\Phi\circ \psi_j(y)$ which are both in the compact set $\overline{B}\subset \mathbb{H}^2$ (the closure of $B$). Then, by the same argument as before, $\alpha^k_{ij}\longrightarrow \alpha_{ij}\in PSL(2,\mathbb{R})$ (up to a subsequence).

\begin{figure}[!h] 
\begin{center}
\includegraphics[height=10cm]{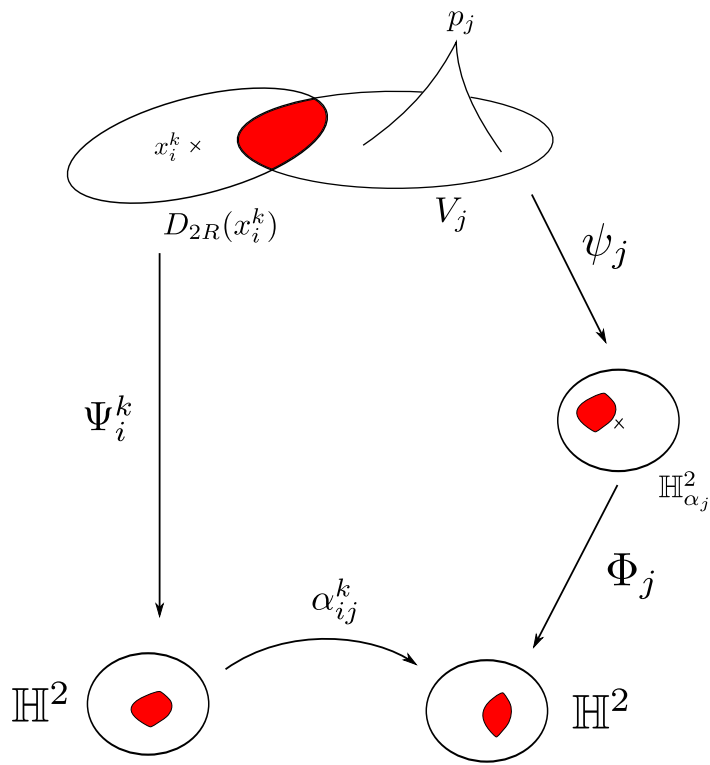}
\end{center}
\caption{The map $\alpha_{ij}^k$} 
\label{alpha}
\end{figure}
Now, define 
$$M:=\left( B_1\sqcup ...\sqcup B_N \sqcup \psi_1(V_1)\sqcup...\sqcup \psi_n(V_n) \right)/\sim,$$
where $B_i=B\subset \mathbb{H}^2$ for each $i$ and $\sim$ identifies:

\begin{itemize}
\item $x_i\in B_i$ with $x_j\in B_j$ whenever $\tau_{ij}$ exists and $\tau_{ij}(x_j)=x_i$.
\item $x_i\in B_i$ with $x_j\in \psi_j(V_j)$ whenever $\alpha_{ij}$ exists and $\alpha_{ij}(x_i)=\Phi(x_j)$.
\end{itemize}

Obviously, $M$ is an hyperbolic surface with cone singularities and defines a point $g\in\M$.

Now, we claim that there exist diffeomorphisms $f_k: M \longrightarrow (\Sigma_\fkp,g_k)$ with $f_k (B_j)\subset D_{R}(x_j^k)$, $f_k(V_i)\subset V_i$ and such that 
$$\Psi^k_j\circ f_k \underset{\mathcal{C}^2}{\longrightarrow} id~\text{ on each } B_j,~\text{ and } \psi_i\circ f_k\underset{\mathcal{C}^2}{\longrightarrow} id~\text{ on  each }\mathbb{H}^2_{\alpha_i}.$$
The proof of this claim is exactly analogous to the proof of \cite[Lemma C4 p.188]{tromba} and will not be repeated here.

Hence, on each $B_j$, we have
$$f_k^*\Psi_j^{k*}g_P\underset{\mathcal{C}^2}{\longrightarrow} g_P,$$
(where $g_P$ is the Poincar\'e metric) and on each $V_i$
$$f_k^*\psi_i^*g_{\alpha_i}\underset{\mathcal{C}^2}{\longrightarrow} g_{\alpha_i}.$$
But, as $\Psi_j^k$ and $\psi_i$ are isometries, we get:
$$f_k^*g_k\underset{\mathcal{C}^2	}{\longrightarrow} g.$$
\end{proof}

Now we are able to prove the properness of $\E_{g_0}$. Let $(\frak{c}_k)_{k\in\mathbb{N}}\subset\T$ such that $(\E_{g_0}(\frak{c}_k))_{k\in\mathbb{N}}$ is convergent. For each $k\in\mathbb{N}$, choose a point $g_k\in\M$ such that the conformal class of $g_k$ is $\frak{c}_k$. It follows that $E(u_{g_k,g_0})\leq K$ for all $k\in\mathbb{N}$.

Let $\gamma$ be a simple closed curve in $\Sigma_\fkp\setminus V$. For $k\in\mathbb{N}$, let $\gamma_k$ be the unique geodesic isotopic to $\gamma$ for the metric $g_k$. Note that there exists no geodesic homotopic to a cone point on a hyperbolic surface. In fact, if $\gamma$ would be such a geodesic, consider the surface obtained by taking two times the connected component of $\Sigma\setminus\gamma$ containing the cone point and glue them along $\gamma$. The remaining surface would be a hyperbolic sphere with two punctures, but it is well-know that such a hyperbolic surface does not exist.

So $\gamma$ is not homotopic to $\partial V_i$ for some $i\in\{1,...,n\}$, so by \cite[Theorem 3.2.4]{tromba} we have:
$$l(\gamma_k)>\frac{C}{K}$$
for some constant $C>0$.

So $(l(\gamma_k))_{k\in\mathbb{N}}$ is bounded from below and we can use Proposition  \ref{mumford} and we get a family $(f_k)_{k\in\mathbb{N}}\subset \text{Diff}(\Sigma_\fkp)$ such that $f_k^*g_k\underset{\mathcal{C}^2}{\longrightarrow} g$.

For all $k\in\mathbb{N}$, denote by $u_k:(\Sigma_\fkp,\frak{c}_k)\longrightarrow (\Sigma_\fkp,g_0)$ the harmonic diffeomorphism isotopic to the identity. By \cite[Lemma 3.2.3]{tromba}, the sequence $(u_k)_{k\in\mathbb{N}}$ is equicontinuous. It follows that the classes of $(f_k)_{k\in\mathbb{N}}$ in $\text{Diff}(\Sigma_\fkp)/\text{Diff}_0(\Sigma_\fkp)$ takes only a finite set of values. In fact, as 
$$E(u_{\frak{c}_k,g_0})=E(u_{f^*_k\frak{c}_k,g_0})=E(f^*_ku_{\frak{c}_k,g_0})<K,$$
 the sequence $(f_k^*u_k)_{k\in\mathbb{N}}$ is equicontinuous and admits a convergent subsequence by Arzel\'{a}-Ascoli. As $\text{Diff}(\Sigma_\fkp)/\text{Diff}_0(\Sigma_\fkp)$ is discrete, there exists a $N\in\mathbb{N}$ such that, for $k$ bigger than $N$, $[f_k]\in \text{Diff}_0(\Sigma_\fkp)$ is constant. It follows that, up to a subsequence, $([f^*_k\frak{c}_k])_{k\in\mathbb{N}}$ converges in $\T$.

\subsection{Weil-Petersson gradient of $\E_{g_0}$}

Let $\frak{c}\in\T$. We are going to use real coordinates $(x,y)$ on $(\Sigma_\fkp,\frak{c})$. From now on, denote by $\partial_1:=\partial_x$ and $\partial_2:=\partial_y$ and by $(dx_1,dx_2)$ the dual framing. Denote by $u:=u_{\frak{c},g_0}$ and fix $\g\in\M$ such that the conformal class of $\g$ is $\frak{c}$. In local coordinates, we have the following expression:
$$du=\sum_{i,j,\alpha,\beta=1}^2 \partial_i u^\alpha dx_{i}\otimes \partial_{u^\alpha},$$
where $(u^1,u^2)$ are the coordinates of $u$ on $(\Sigma_\fkp,g_0)$. Assume that $(u^1,u^2)$ are isothermal coordinates for $g_0$, so 
$$g_0= \sum_{\alpha, \beta=1}^2 \sigma^2(u)\delta_{\alpha\beta}du^\alpha du^\beta,$$
(here $\delta_{\alpha\beta}$ is the Kronecker symbol). Writing $\g$ in coordinates and using the Einstein convention, we have the following expression:
$$E(u)=\frac{1}{2}\int_{\Sigma_\fkp} \| du \|^2dv_{\widetilde{g}}=\frac{1}{2}\int_{\Sigma_\fkp}\sigma^2\delta_{\alpha\beta}\g^{ij}\partial_iu^\alpha\partial_ju^\beta vol_{\g}.$$
Here, $vol_{\g}$ is the volume form of $(\Sigma_\fkp,\g)$ and $\g^{ij}$ are the coefficients of the metric dual to $\g$ in $T^*\Sigma_\fkp$.

For $h\in T_\frak{c}\T$, denote by $\h$ the horizontal lift of $d\Theta_\alpha(h)$ in $T_{\g}\M$ (recall that $\Theta_\alpha$ is the application given by the uniformization). So $\widetilde{h}$ is a zero trace divergence-free symmetric $2-$tensor on $(\Sigma_\fkp,\widetilde{g})$.

We are going to compute the differential of $\widetilde{\E}_{g_0}$ at $\widetilde{g}$ in the direction $\h$. Note that the differential of $\g\longmapsto (\g^{ij})$ is given by $\h\longmapsto (-\h^{ij})$ and the differential of $\g\longmapsto vol_{\g}$ is $\h\longmapsto (\frac{1}{2}\tr_{\g}\h)vol_{\g}$. So one gets:

$$d \widetilde{\E}_{g_0}(\g)(\h)=-\frac{1}{2}\int_{\Sigma_\fkp}\sigma^2\h^{ij}\partial_iu^\alpha\partial_ju^\alpha vol_{\g}+\frac{1}{4}\int_{\Sigma_\fkp}\sigma^2\g^{ij}\partial_iu^\alpha\partial_ju^\alpha(tr_{\g}\h)vol_{\g}+R(\h),$$
where the term $R(\h)$ is obtained by fixing $\widetilde{g}$ and $dvol_{\widetilde{g}}$ and varying the rest. It follows that $R(\h)$ correspond to the first order variation of $E(u)$ in the direction $\widetilde{h}$. But as $u$ is harmonic, $R(\widetilde{h})=0$. 

Moreover, the second term is zero because we have chosen a horizontal lift of $h$, hence $tr_{\g}\h=0$.

Writing $u=u^1+iu^2$ and using the fact that $\h^{11}=-\h^{22}$ and $\h^{12}=\h^{21}$ (see Section \ref{teich}), we get the following expression:
\begin{eqnarray*}
d\E_{g_0}(\g)(\h) & = & -\frac{1}{2}\int_{\Sigma_\fkp}\sigma^2\left( \h^{11}\left(\vert \partial_1u\vert^2-\vert\partial_2u\vert^2\right) + 2\h^{12}\Re(\partial_1u\partial_2 \overline{u})\right)vol_{\g}\\
& = & \langle \h,\varphi\rangle_{\Sym^2(\Sigma_\fkp)},
\end{eqnarray*}
where
$$\varphi=-\frac{1}{2}\sigma^2(u)\left((\vert\partial_1u\vert^2-\vert \partial_2u\vert^2)(dx^2-dy^2)+2\Re(\partial_1u\partial_2\overline{u})(dxdy+dydx)\right).$$
Note that, by definition, $\varphi$ is the Weil-Petersson gradient $\nabla\E(\frak{c})$ of $\E$ at the point $\frak{c}\in\T$. 
On the other hand,
\begin{eqnarray*}
\Re(\Phi(u)) & = & \Re(\sigma^2(u)\partial_zu\partial_z\overline{u}dz^2) \\
& = & \Re\left(\frac{1}{4}\sigma^2(u)(\partial_1u-i\partial_2u)(\partial_1\overline{u}-i\partial_2\overline{u})(dx^2-dy^2+i(dxdy+dydx))\right) \\
& = & \frac{1}{4}\sigma^2(u)\left((\vert\partial_1u\vert^2-\vert\partial_2u\vert^2)(dx^2-dy^2)+2\Re(\partial_1u\partial_2\overline{u})(dxdy+dydx)\right).
\end{eqnarray*}
So $\nabla\E(\frak{c})=-2\Re(\Phi(u))$.

\section{Minimal diffeomorphisms between hyperbolic cone surfaces}\label{minimal}

In this section, we prove the Main Theorem by studying the PDE satisfied by harmonic diffeomorphisms.

\subsection{Existence}

\begin{prop}\label{existence}
For each $\alpha,\alpha'\in \left(0,\frac{1}{2}\right)^n$, $g_1\in \F$ and $g_2\in\FF$, there exists a minimal diffeomorphism $\Psi: (\Sigma_\fkp,g_1)\longrightarrow (\Sigma_\fkp,g_2)$ isotopic to the identity.
\end{prop}

\begin{proof}
Let $g_1\in\F$, $g_2\in\FF$ and consider $M:=(\Sigma_\fkp\times\Sigma_\fkp,g_1\oplus g_2)$. Given a conformal structure $\frak{c}\in\T$, one can consider the map
$$f_\frak{c}:=(u_1,u_2) : (\Sigma_\fkp,\frak{c})\longrightarrow M,$$
where $u_i: (\Sigma_\fkp,\frak{c})\longrightarrow (\Sigma_\fkp,g_i)$ is the harmonic diffeomorphism isotopic to the identity ($i=1,2$).

Clearly, $E(f_\frak{c})=E(u_1)+E(u_2)$. From Section \ref{energyfunctional}, the functional $\E:=\E_{g_1}+\E_{g_2}: \T\longrightarrow \mathbb{R}$  is proper. Let $\frak{c}_0$ be a critical point of $\E$, so the map $\Psi:=f_{\frak{c}_0}:~(\Sigma,\frak{c}_0)\longrightarrow M$ is a harmonic immersion. We claim that $\Psi$ is also conformal. In fact, $\Psi=(u_1,u_2)$, so
\begin{eqnarray*}
\Psi^*(g_1\oplus g_2) & = & u_1^*g_1\oplus u_2^*g_2 \\
& = & \Phi(u_1)+\Phi(u_2)+\rho^2(z)(e(u_1)+e(u_2))\vert dz\vert^2 + \overline{\Phi(u_1)}+\overline{\Phi(u_2)},
\end{eqnarray*}
where $z$ is a local holomorphic coordinates on $(\Sigma_\fkp,\frak{c}_0)$ such that $\Theta_\alpha(\frak{c}_0)=\rho^2(z)\vert dz\vert^2$.

Now, as $\frak{c}_0$ is a minimum of $\mathcal{E}$, $\nabla\mathcal{E}(\frak{c}_0)=-2\Re\left(\Phi(u_1)+\Phi(u_2)\right)=0$, so $\Phi(u_1)+\Phi(u_2)=0$ and $\Psi$ is conformal. It follows that $\Psi$ is a conformal harmonic immersion, hence $\Psi(\Sigma_\fkp)$ is a minimal surface in $M$ (see \cite[Proposition p. 119]{eells-sampson}).

Denoting by $p_i : M \longrightarrow \Sigma_\fkp$ the projection on the $i$-th factor ($i=1,2$) and $\Gamma=\Psi(\Sigma_\fkp)$, we get that $u_i=p_{i_{\vert \Gamma}}$ and $\Gamma=\text{graph}(p_{2_{\vert\Gamma}}\circ p^{-1}_{1_{\vert\Gamma}})$. It follows that $$p_{2_{\vert\Gamma}}\circ p_{1_{\vert\Gamma}}^{-1}:(\Sigma_\fkp,g_1)\longrightarrow (\Sigma_\fkp,g_2)$$
is a minimal diffeomorphism isotopic to the identity.
\end{proof}

\begin{rem}
For $\Psi: (\Sigma,g_1)\longrightarrow (\Sigma,g_2)$ a minimal diffeomorphism as in Proposition \ref{existence}, the induced metric $g_\Gamma$ on $\Gamma=\text{graph}(\Psi)$ carries conical singularities of angle $\beta=(\beta_1,...,\beta_n)$ where $\beta_i=2\pi\min(\alpha_i,\alpha'_i)$. In fact, normalizing the metrics $g_1$ and $g_2$ so that $\Psi=id$, and choosing conformal coordinates $z$ in a neighborhood of $(p_i,p_i)\in \Gamma$, one has the following expression:
\begin{eqnarray*}
g_\Gamma & = & id^*g_1+id^*g_2 \\
& = & \left( e^{2\lambda}\vert z \vert^{2(\alpha-1)}+e^{2\lambda'}\vert z\vert^{2(\alpha'-1)}\right)\vert dz\vert^2 \\
& = & e^{2\lambda}\vert z \vert^{2(\alpha-1)}\left( 1+ e^{2(\lambda'-\lambda)}\vert z\vert^{2(\alpha'-\alpha)}\right) \vert dz\vert^2 \\
& = & e^{2\mu}\vert z \vert^{2(\alpha-1)}\vert dz\vert^2
\end{eqnarray*}
where $\lambda,\lambda'$ and $\mu$ are continuous functions. So $g_\Gamma$ carries a conical singularity of angle $2\pi\alpha$ at $(p_i,p_i)$ in the sense of Remark \ref{generalsingularity}.
\end{rem}

\subsection{Uniqueness}

Before proving the rest of the Main Theorem, let's recall some results about the harmonic diffeomorphisms provided by \cite{jesse}. We use the same notations as in the proof above. Let $z$ be conformal coordinates on $\Gamma$ such that $$g_\Gamma=\rho^2(z)\vert dz\vert^2,~g_i=\sigma_i^2(u_i(z))\vert du_i\vert^2.$$
The natural complex structure on $\Gamma$ provides a decomposition of vector-valued $1$-forms on $\Gamma$ into their $\mathbb{C}$-linear and $\mathbb{C}$-antilinear part. In particular, for $i=1,2$, we get:
$$\frac{1}{\sqrt{2}}du_i=\partial u_i +\overline{\partial}u_i,$$
where $\partial u_i\in \Omega^{1,0}(\Gamma, u_i^*T\Sigma_\fkp),~\overline{\partial}u_i\in \Omega^{0,1}(\Gamma,u_i^*T\Sigma_\fkp)$.
In follows that 
$$e(u_i)=\frac{1}{2}\Vert du_i\Vert^2=\Vert \partial u_i\Vert^2 + \Vert \overline{\partial} u_i \Vert^2,$$
which in coordinates gives
$$
\left\{
   \begin{array}{ll}
        \|\partial u_i \|^2(z)=\rho^{-2}(z)\sigma_i^2(u_i(z))\vert\partial_z u_i\vert^2 \\
        \|\overline{\partial} u_i \|^2(z)=\rho^{-2}(z)\sigma_i^2(u_i(z))\vert\overline{\partial}_z u_i\vert^2.
   \end{array}
\right.
$$
Then we have the following expressions (cf. Section \ref{energyfunctional}):
$$\left\{
\begin{array}{l}
  \|\Phi(u_i)\|=\|\partial u_i \| \|\overline{\partial}u_i\| \\
  e(u_i)=\|\partial u_i \|^2+\|\overline{\partial}u_i\|^2 \\
  J(u_i)= \|\partial u_i \|^2-\| \overline{\partial}u_i\|^2.
\end{array}
\right.$$
Note that, as $u_i$ is orientation preserving, $J(u_i)>0$ and in particular $\|\partial u_i \|\neq 0$.

It is well-known that these functions satisfy a Bochner type identities everywhere it is defined (see \cite{schoenyau})
\begin{equation}\label{bochner}
  \left\{
      \begin{aligned}
       & \Delta \ln \|\partial u_i \| =\|\partial u_i \|^2-\| \overline{\partial}u_i\|^2-1\\        
       & \Delta \ln \|\overline{\partial} u_i \| =-\|\partial u_i \|^2+\| \overline{\partial}u_i\|^2-1,\\
      \end{aligned}
    \right.
\end{equation}
where $\Delta=\Delta_{g_\Gamma}=\delta\delta^*$.

Note that, as $\Phi(u_i)$ is holomorphic outside $\fkp$, the singularities of $\ln \| \overline{\partial} u_i \|$ on $\Sigma_\fkp$ are isolated and have the form $c\ln r$ for some $c>0$. In fact, as $J(u_i)>0$, $\|\partial u_i\|\neq 0$. Because $\|\Phi(u_i)\|=\|\partial u_i \| \|\overline{\partial}u_i\|$, the singularities of $\ln \| \overline{\partial} u_i \|$ correspond to zeros of $\Phi(u_i)$.

Now, let's describe the behavior of $\|\partial u_i \|$ and $\|\overline{\partial} u_i \|$ around a puncture. Let $z$ be a conformal coordinates system on $(\Sigma_\fkp,g_\Gamma)$ centered at $p$. From \cite[Section 2.3, Form 2.3]{jesse}, in a neighborhood $U$ of a puncture of angle $2\pi\alpha$, the map $u_i$ has the following form:
$$u_i(z)=\lambda_i z + r^{1+\epsilon}f_i(z),$$
where $\lambda_i\in \mathbb{C}^*$, $r=\vert z \vert$, $\epsilon>0$ and $f\in\chi^{2,\gamma}_b(U)$ (see Subsection \ref{hyperbolicconesurfaces}). Note that in particular, $f\in \mathcal{C}^2(U)$. Using 
$$
\left\{
   \begin{array}{ll}
       \partial_z = & \frac{1}{2z}(r\partial_r-i\partial_\theta) \\
       \overline{\partial}_z = & \frac{1}{2\overline{z}}(r\partial_r+i\partial_\theta)
   \end{array}
\right.
$$
we get that
$$
\left\{
   \begin{array}{ll}
       \partial_z u_i = & \lambda_i + r^\epsilon L(f_i) \\
       \overline{\partial}_zu_i = & r^\epsilon\overline{L}(f_i)
   \end{array}
\right.
$$
where 
$$
\left\{
   \begin{array}{ll}
       L=\frac{r}{2z}\big( (1+\epsilon)Id+\partial_r-i\partial_\theta\big) \\
       \overline{L}=\frac{r}{2\z}\big( (1+\epsilon)Id+\partial_r+i\partial_\theta\big).
   \end{array}
\right.
$$
Let $\alpha$ (resp. $\alpha'$) be the cone angle of the singularity of $g_1$ (resp. $g_2$) at $p$. So, from section \ref{sing}, there exists some bounded non vanishing functions $c_1$ and $c_2$ so that
$$
\left\{
   \begin{array}{ll}
       \sigma^2_1(u_1)=c_1^2\vert u_1 \vert^{2(\alpha-1)} \\
       \sigma^2_2(u_2)=c_2^2\vert u_2 \vert^{2(\alpha'-1)}.
   \end{array}
\right.
$$
It follows that
\begin{equation}\label{equation3}
\left\{
   \begin{array}{ll}
       \| \partial u_1 \|^2 & =  \rho^{-2}(z)c_1^2\vert\lambda_1z+r^{1+\epsilon}f_1\vert^{2(\alpha-1)}\vert\lambda_1+r^\epsilon L(f_1)\vert^2 \\
       & =  \rho^{-2}(z)c_1^2\vert\lambda_1\vert^{2\alpha}r^{2(\alpha-1)} \left(1+O(r^\epsilon)\right) \\
       \| \overline{\partial}u_1 \|^2 & =  \rho^{-2}(z)c_1^2 \vert \lambda_1 \vert^{2(\alpha-1)}r^{2(\alpha-1)+2\epsilon}\vert \overline{L}(f_1)\vert^2 (1+O(r^\epsilon)).
   \end{array}
\right.
\end{equation}

\begin{prop}\label{uniqueness}
If $\alpha_i<\alpha'_i$ for all $i\in \{1,...,n\}$, the minimal diffeomorphism $\Psi: (\Sigma_\fkp,g_1)\longrightarrow (\Sigma_\fkp,g_2)$ of Proposition \ref{existence} is unique.
\end{prop}
The proof follows from the stability of $\Gamma$.

\begin{lemma}\label{graphstability}
Under the same conditions as in Proposition \ref{uniqueness}, a minimal graph $\Gamma\in (\Sigma_\fkp\times\Sigma_\fkp,g_1\oplus g_2)$ is stable.
\end{lemma}

\begin{proof}
Let $\Gamma$ be a minimal graph in $(\Sigma_\fkp\times\Sigma_\fkp,g_1\oplus g_2)$, and denote by $u_i$ the $i^{th}$ projection from $\Gamma$ to $(\Sigma,g_i)$ (for $i=1,2$). As $\Gamma$ is minimal, the $u_i$ are harmonic and $\Phi(u_1)+\Phi(u_2)=0$.

Stability of minimal graph in products of surfaces has been studied for the classical case in \cite{wan}. We have the following lemma:
\begin{lemma}\label{variationarea}
Let $\Gamma$ be a minimal graph in $(\Sigma_\fkp\times\Sigma_\fkp,g_1\oplus g_2)$, then the second variation of the area functional under a deformation of $\Gamma$ fixing its intersection with the singular loci is given by:
\begin{equation}\label{area}
A''=E_2''-4\int_\Gamma\frac{\|\Phi'(u_2)\|^2}{e(u_1)+e(u_2)} vol_\Gamma,
\end{equation}
where $E''_2$ is the second variation of the energy of $u_2$ and $\Phi'(u_2)$ is the variation of the Hopf differential of $u_2$.
\end{lemma}

\begin{proof}
By definition, the area of $\Gamma$ is given by:
$$A=\int_\Gamma \left(\det(u_1^*g_1\oplus u_2^*g_2)\right)^{1/2}\vert dz\vert^2.$$
But we have:
\begin{eqnarray*}
\det(u_1^*g_1\oplus u_2^*g_2) & = & \det\left(\rho^2(e(u_1)+e(u_2))\vert dz\vert^2+2\Re(\Phi(u_1)+\Phi(u_2))\right) \\
& = & \det\big(\rho^2(e(u_1)+e(u_2))(dx^2+dy^2)+2\Re(\phi(u_1)+\phi(u_2))(dx^2-dy^2)- \\
& - & 2\Im(\phi(u_1)+\phi(u_2))(dxdy+dydx)\big) \\
& = & \rho^4(e(u_1)+e(u_2))^2-4\vert \phi(u_1)+\phi(u_2)\vert^2,
\end{eqnarray*}
where $\Phi(u_i)=\phi(u_i)dz^2$.
It follows that
$$A=\int_\Gamma \left (e(u_1)+e(u_2))^2-4\| \Phi(u_1)+\Phi(u_2)\|^2\right)^{1/2}dv_\Gamma.$$
Writing
$$a:=(e(u_1)+e(u_2))^2-4\|\Phi(u_1)+\Phi(u_2)\|^2,$$
we get 
$$A=\int_\Gamma a^{1/2}dv_\Gamma.$$
Recall that, for $i=1,2$, we have
$$E(u_i)=\int_{\Sigma_\fkp} e(u_i)dv_{\Gamma}.$$

Denote by $v_{1,t}$ and $v_{2,t}$ be the variations of $u_1$ and $u_2$ respectively corresponding to a variation $\Gamma_t$ of $\Gamma$. Set $\psi_i:=\frac{d}{dt}_{\vert t=0} v_{i,t}$ which is a section of $u_i^*T\Sigma_\fkp$. Denote by $\nabla^{u_i}$ the pull-back by $u_i$ of the Levi-Civita connection on $(\Sigma_\fkp,g_i)$. In particular, we have:
$$\frac{d}{dt}_{\vert t=0} dv_{i,t}=\nabla^{u_i}\psi_i.$$
Now we have:
$$A''(\Gamma) = \displaystyle{\frac{d^2}{dt^2}_{\vert t=0}}\int_\Gamma a_t^{1/2}dv_\Gamma=\frac{1}{2} \int_\Gamma (a^{-1/2}a''-\frac{1}{2}a^{-3/2}a'^2)dv_\Gamma.$$
But 
\begin{eqnarray*}
a' & = & \displaystyle{\frac{d}{dt}_{\vert t=0}} \left( (e(v_{1,t})+e(v_{2,t}))^2-4(\| \Phi(v_{1,t})+\Phi(v_{2,t})\|^2\right) \\
& = & 2(e(u_1)+e(u_2))(e'(u_1)+e'(u_2))-8\langle \Phi'(u_1)+\Phi'(u_2),\Phi(u_1)+\Phi(u_2)\rangle \\
& = & 2(e(u_1)+e(u_2))(e'(u_1)+e'(u_2)),
\end{eqnarray*}
and 
\begin{eqnarray*}
a'' & = & \displaystyle{\frac{d^2}{dt^2}_{\vert t=0}} \left( (e(v_{1,t})+e(v_{2,t}))^2-4(\| \Phi(v_{1,t})+\Phi(v_{2,t})\|^2\right) \\
& = & 2(e'(u_1)+e'(u_2))^2+2(e(u_1)+e(u_2))(e''(u_1)+e''(u_2))-8\|\Phi'(u_1)+\Phi'(u_2)\|^2.
\end{eqnarray*}
Hence,
$$a^{-1/2}a''-\frac{1}{2}a^{-3/2}a'^2= 2(e''(u_1)+e''(u_2))-8\frac{\|\Phi'(u_1)+\Phi'(u_2)\|^2}{e(u_1)+e(u_2)}.$$
It follows

$$A''(\Gamma) = E''(u_1)+E''(u_2)-4\int_\Gamma \frac{\|\Phi'(u_1)+\Phi'(u_2)\|^2}{e(u_1)+e(u_2)}dv_\Gamma.$$

Now, as pointed out in \cite{wan}, such a variation can be realized as a variation of $u_2$ only since the variation of $u_1$ can be interpreted as a change of coordinates which does not change the area functional. So, setting $\psi_1=0$, we get the formula.
\end{proof}

Writing $\displaystyle{w_i:=\ln \frac{\|\partial u_i \|}{\|\overline{\partial}u_i\|}}$ and using equation (\ref{bochner}), we obtain:
\begin{eqnarray*}
\Delta w_i & = & \Delta \ln\|\partial u_i \| - \Delta \ln \| \overline{\partial} u_i\| \\
& = & 2\|\partial u_i \|^ 2-2\|\overline{\partial}u_i\|^2 \\
& = & 2\|\Phi\|\left( \frac{\| \partial u_i \|}{\| \overline{\partial}u_i\|} - \left(\frac{\| \partial u_i \|}{\| \overline{\partial}u_i\|}\right)^{-1}\right) \\
& = & 4\| \Phi\| \sinh w_i,
\end{eqnarray*}
where $\|\Phi\|=\|\Phi(u_1)\|=\|\Phi(u_2)\|$. That is, $w_1$ and $w_2$ satisfy the same equation. Note that, outside $\fkp$, the singularities of $w_1$ and $w_2$ are the same. In fact, singularities of $w_i$ correspond to zeros of $\| \partial u_i\|$ (as $J(u_i)=\|\partial u_i \|^2-\|\overline{\partial}u_i \|^2>0$). But as $\|\Phi(u_1)\|=\|\partial u_1\|\|\overline{\partial}u_1\|=\|\partial u_2\|\|\overline{\partial}u_2\|$, the zeros of $\|\partial u_1\|$ and $\|\partial u_2\|$ are the same. In particular, $w_2-w_1$ is a regular function on $\Sigma_\fkp$ satisfying:
\begin{equation}\label{Bochner}
\Delta (w_2-w_1) = 4\| \Phi \| (\sinh w_2 -\sinh w_1).
\end{equation}
Let's study the behavior of $w_1-w_2$ at a singularity $p\in\fkp$. Using the same notation as above,  the norm of the Hopf differentials satisfy:
\begin{eqnarray*}
\rho^2(z)\| \Phi(u_1)\|(z) & = & \sigma_1^2(u_1)\vert\partial_zu_1\vert\vert\partial_z\overline{u}_1\vert \\
& = & c_1^2\vert \lambda_1 z +r^{1+\epsilon}f\vert^{2(\alpha-1)} \vert \lambda_i + r^\epsilon L(f_1)\vert \vert r^\epsilon \overline{L}(f_1)\vert \\
  & = & c_1^2\vert \overline{L}(f_1)\vert \vert \lambda_i\vert^{2\alpha-1} r^{2(\alpha-1)+\epsilon}(1+O(r^\epsilon))
\end{eqnarray*} 
and

\begin{eqnarray*}
\rho^2(z)\| \Phi(u_2)\|(z) & = & \sigma_1^2(u_2)\vert\partial_zu_2\vert\vert\partial_z\overline{u}_2\vert \\
& = & c_2^2\vert \lambda_2 z +r^{1+\epsilon}f_2\vert^{2(\alpha'-1)} \vert \lambda_2 + r^\epsilon L(f_2)\vert \vert r^\epsilon \overline{L}(f_2)\vert \\
 & = & c_2^2\vert \overline{L}(f_2)\vert \vert \lambda_i\vert^{2\alpha'-1} r^{2(\alpha'-1)+\epsilon}(1+O(r^\epsilon)).
\end{eqnarray*}
Hence, using $\|\Phi(u_1)\|=\|\Phi(u_2)\|$,

$$\left\vert \frac{\overline{L}(f_1)}{\overline{L}(f_2)}\right\vert=r^{2(\alpha'-\alpha)}C,$$
where $C$ is a non-vanishing bounded function.
Now, using equation (\ref{equation3}), we obtain:

$$w_i  =  \ln\left( \frac{\vert \lambda_i\vert}{r^\epsilon\vert \overline{L}(f_i)\vert}(1+ O(r^\epsilon))\right)  =  \ln\left( \frac{\vert\lambda_i\vert}{r^\epsilon\vert \overline{L}(f_i)\vert}\right) + O(r^\epsilon).$$
In particular,
\begin{equation}\label{behavioratpuncture}
w_2-w_1=2(\alpha-\alpha')\ln r + C',
\end{equation}
where $C'$ is a bounded function.
As $\alpha-\alpha'> 0$, $w_2-w_1$ tends to $-\infty$ at the singularities.

So we can apply the maximum principle to equation (\ref{Bochner}), and we obtain that $w_2\leq w_1$. Using $\| \Phi(u_1)\|=\|\Phi(u_2)\|=\| \Phi\|$, we finally obtain:
$$\| \partial u_2 \| \leq \| \partial u_1\|.$$
Let's consider the function $\displaystyle{f(x)=x+\| \Phi\|^2x^{-1}}$ defined on $\mathbb{R}_{>0}$. Its derivative is $\displaystyle{f'(x)=1-\|\Phi\|^2x^{-2}}$, so $f$ is increasing for $\displaystyle{x\geq \|\Phi\|}$.
As $J(u_2)>0$,
$$\| \partial u_2\|^2\geq \| \partial u_2 \| \|\overline{\partial}u_2 \|=\frac{\| \Phi \|}{2}.$$
Applying $f$ to $\| \partial u_2 \|^2\leq \| \partial u_1 \|^2$, we get
$$e(u_2)\leq e(u_1).$$
So, from equation (\ref{area}), we obtain:
$$A''\geq E''_2-2\int_\Omega\frac{\|\Phi'(u_2)\|^2}{e(u_2)}vol_\Gamma.$$
Let $\psi:=\frac{d}{dt}_{\vert t=0}v_t$ be a deformation of $u_2$ (so $\psi$ is a section of $u_2^*T\Sigma_\fkp$). We have the following expression (see e.g \cite[Equation 2]{smith}):

$$E''(u_2)=\int_{\Gamma}\left( \langle \nabla^{u_2}\psi,\nabla^{u_2}\psi\rangle-tr_{g_\Gamma}R^{g_2}(du_2,\psi,\psi,du_2)\right)dv_\Gamma,$$
where $R^{g_2}$ is the curvature tensor on $(\Sigma_\fkp,g_2)$, $\nabla^{u_2}$ is the pull-back by $u_2$ of the Levi-Civita connection on $(\Sigma_\fkp,g_2)$ and the scalar product is taken with respect to the metric $g_\Gamma^*\otimes u_2^*g_2$ on $T^*\Gamma\otimes u_2^*T\Sigma_\fkp$.
Computing $\Phi'$, we get:
\begin{eqnarray*}
 \Phi'  & = & \displaystyle{\frac{d}{dt}_{\vert t=0}} v_t^*g_2(\partial_z,\partial_z)dz^2 \\
 & = & \displaystyle{\frac{d}{dt}_{\vert t=0}} g_2(dv_t(\partial_z),dv_t(\partial_z))dz^2 \\
 & = & 2 g_2(\nabla^{u_2}\psi(\partial_z),du_2(\partial_z))dz^2.
\end{eqnarray*}
That is
$$\| \Phi'\| ^2 = 4\sigma^2(u_2)\vert\langle\nabla^{u_2}\psi(\partial_z),du_2(\partial_z)\rangle\vert^2,$$
(where $\langle.,.\rangle$ is the scalar product with respect to $g_2$). By Cauchy-Schwarz and equation (\ref{normcomputation}), we get
\begin{eqnarray*}
\| \Phi'\| ^2 & \leq & 4\sigma^2(u) \left\vert \langle \nabla^{u_2}\psi(\partial_z),\overline{\nabla^{u_2}\psi(\partial_z)}\rangle\right\vert\left\vert\langle du_2(\partial_z),\overline{du_2(\partial_z)}\rangle\right\vert \\
& \leq & \frac{1}{4}\|\nabla^{u_2}\psi\|^2 \|du_2\|^2.
\end{eqnarray*}
Hence,
$$\int_{\Gamma}\frac{\|\Phi'\|^2}{e(u_2)}vol_\Gamma\leq\frac{1}{2}\int_\Gamma \langle \nabla^u\psi,\nabla^u\psi\rangle vol_\Gamma.$$
Finally, we obtain:
$$A''\geq -\int_\Gamma tr_{g_\Gamma}R^{g_2}(du,\psi,\psi,du)dv_\Gamma.$$
But as the sectional curvature of $(\Sigma_\fkp,g_2)$ is $-1$, the right-hand side of the last equation is strictly positive (for a non zero $\psi$). So $\Gamma$ is strictly stable.
\end{proof}
Now, using the classical estimates (see \cite[Proposition p.126]{eells-sampson} or the proof of lemma \ref{variationarea}),
$$\text{Area}(\Gamma)\leq E(\Psi)$$
and equality holds if and only if $\Psi$ is a minimal immersion. It follows from the stability of $\Gamma$ that the critical points of $\E_{g_1}+\E_{g_2}$ can only be minima. But a proper function whose unique extrema are minima with non-degenerate Hessian admits a unique minimum. So $\Psi$ is the unique minimal diffeomorphism isotopic to the identity.

\bibliographystyle{alpha}
\bibliography{Minimal diffeomorphism between singular hyperbolic surfaces.bbl}

\end{document}